\documentclass[10pt]{amsart}

\usepackage[english]{babel}
\usepackage[T1]{fontenc}
\usepackage{amsmath,amssymb,stmaryrd}
\usepackage{color}
\usepackage{mathdots}
\usepackage{tikz}
\usepackage[all]{xy}
\usepackage{graphicx}
\usepackage{fancyhdr}
\usepackage{hyperref} 

 \theoremstyle{plain}
 \newtheorem{lemma}{Lemma}[section]
\newtheorem{theorem}[lemma]
{Theorem }
\newtheorem{corollary}[lemma]
{Corollary}
\newtheorem{prop}[lemma]{Proposition}

\theoremstyle{definition}
\newtheorem{definition}[lemma]{Definition }
\newtheorem{rmkdef}[lemma]{Remark-Definition}

\newtheorem{ex}[lemma]
{Example }

\newtheorem{rmk}[lemma]{Remark}

\newcommand{\lgw}{\longrightarrow}
\newcommand{\lgm}{\longmapsto}
\newcommand{\s}{\sigma}

\newcommand{\h}{\operatorname{H}}
\newcommand{\Deg}{\operatorname{Deg}}

\renewcommand{\deg}{\operatorname{deg}}
\newcommand{\ord}{\operatorname{ord}}

\newcommand{\wdh}{\widehat}
\newcommand{\ini}{\operatorname{in}}
\newcommand{\wdt}{\widetilde}
\renewcommand{\l}{\lambda}

\newcommand{\m}{\mathfrak{m}}
\newcommand{\Supp}{\operatorname{Supp}}

\newcommand{\Res}{\operatorname{Res}}

\newcommand{\het}{\operatorname{ht}}

\newcommand{\co}{\text{co}}

\newcommand{\lag}{\langle}
\newcommand{\rag}{\rangle}
\renewcommand{\k}{\Bbbk}

\renewcommand{\dim}{\operatorname{dim}}

\newcommand{\cha}{\text{char\,}}
\newcommand{\R}{\mathbb{R}}
\newcommand{\K}{\mathbb{K}_{n-1}}
\newcommand{\KK}{\mathbb{K}}

\newcommand{\N}{\mathbb{N}}

\newcommand{\C}{\mathbb{C}}
\newcommand{\und}{\underline}
\newcommand{\p}{\mathfrak p}
\newcommand{\Q}{\mathbb{Q}}

\newcommand{\D}{\Delta}
\newcommand{\lb}{\llbracket}
\newcommand{\rb}{\rrbracket}

\renewcommand{\a}{\alpha}

\renewcommand{\b}{\beta}
\newcommand{\g}{\gamma}

\renewcommand{\phi}{\varphi}

\newcommand{\e}{\varepsilon}
\begin{document}
\title[Effective Division of algebraic power series]{Local zero estimates  and effective division in rings of algebraic power series }

\author{Guillaume Rond}

\email{guillaume.rond@univ-amu.fr}
\address{Aix-Marseille Universit\'e, CNRS, Centrale Marseille, I2M, UMR 7373, 13453 Marseille, France}

\thanks{The author was partially supported by ANR projects STAAVF (ANR-2011 BS01 009) and SUSI (ANR-12-JS01-0002-01)}

\begin{abstract}
We give a necessary condition for algebraicity of  finite modules over the ring of formal power series. This condition is given in terms of local zero estimates. In fact we show that this condition is also sufficient when the module is a ring with some additional properties. To prove this result we show an effective Weierstrass Division Theorem and an effective solution to the Ideal Membership Problem in rings of algebraic power series. Finally we apply these results to prove a gap theorem for power series which are remainders of the Grauert-Hironaka-Galligo Division Theorem.
\end{abstract}

\subjclass[2010]{Primary : 13J05, Secondary :   11G50, 11J82, 13P10}

\maketitle
\tableofcontents
\section{Introduction}

The goal of this paper is to give a necessary condition in term of local zero estimates for a finite module defined over the ring of formal power series to be the completion of a module defined over the ring of algebraic power series.  Finding conditions for the algebraicity of such modules is a long-standing problem (see \cite{Sa} or \cite{Ar} for instance).  Let us recall that an algebraic power series over a field $\k$ in the variables $x_1$, $\cdots$, $x_n$ is a formal power series $f(x)\in\k \lb x\rb$ (from now on we denote the tuple $(x_1,\cdots,x_n)$ by $x$) such that 
$$P(x,f(x))=0$$
for a non-zero polynomial $P(x,T)\in\k[x,T]$. The set of algebraic power series is a subring of $\k\lb x\rb$ denoted by $\k\lag x\rag$.
\\
For an algebraic power series $f$, we define the height of $f$, $\h(f)$,  to be the maximum of the degrees of the coefficients of the minimal polynomial of $f$  (see Definition \ref{height}). If $f$ is a polynomial its height is equal to its degree as a polynomial.\\
 Let $M$ be a  $\k\lb x\rb$-module The order function $\ord_M$ is defined as follows:
$$\ord_M(m):=\sup\{c\in\N\ / \ m\in (x)^cM\}\ \ \forall m\in M\backslash \{0\}.$$
 Let $p\in\k[x]^s$ (resp. $\k\lag x\rag^s$). The \emph{degree}  (resp. \emph{height}) of $p$ is the maximum of the degrees (resp. heights) of its components. Then our main result is the following:

\begin{theorem}\label{main} Let $\k$ be any  field and let $M$ be a finite $\k\lb x\rb$-module, 
$$M=\frac{\k\lb x\rb^s}{N}$$ for some integer $s$ and some $\k\lb x\rb$-sub-module $N$ of $\k\lb x \rb^s$.
Let us assume that 
the sub-module $N$ is generated by a $\k\lag x\rag$-sub-module of $\k\lag x\rag^s$. Then there exists a function
$$C : \N \lgw \R_{>0}$$
 such that
\begin{equation}\label{lze1}\ord_M(f)\leq C(\Deg(f))\cdot \h(f) \ \ \ \forall f\in \k\lag x\rag^s\backslash N.\end{equation}
Here $\Deg(f)$ denotes the degree of the field extension $\k(x)\lgw \k(x,f)$. Moreover when $\cha(\k)=0$ then $C$ depends polynomially on $\Deg(f)$.
\end{theorem}
\vspace{0.2cm}
\begin{corollary}\label{izu}
With the notations of Theorem \ref{main}, let us assume that $N$ is generated by a $\k\lag x\rag$-sub-module of $\k\lag x\rag^s$. Then there exists a constant $C'>0$ such that 
\begin{equation}\label{lze2}\ord_{M}(p)\leq C'\cdot\deg(p) \ \ \ \forall p\in \k[x]^s\backslash N.\end{equation}
\end{corollary}
\begin{proof}
Indeed, for a vector of polynomials $p\in\k[x]^s$ we have $\Deg(p)=1$ and $\h(p)=\deg(p)$, so the inequality is satisfied with $C'=C(1)$ where $C$ is the  function of Theorem \ref{main}.
\end{proof}
\vspace{0.2cm}
\noindent We also prove a partial converse of Corollary \ref{izu}:\\

\begin{theorem}\label{main2}
Let $R$ be a ring of the form $\frac{\k\lb x\rb}{I}$ for some ideal $I$ such that
$$I=P_1^{n_1}\cap\cdots\cap P_l^{n_l}$$
where the $P_i$ are prime ideals with $\het(P_i)=\het(P_j)$ for all $i$ and $j$, and the $n_i$ are positive integers.\\
If there exists a constant $C>0$ such that 
$$\ord_{R}(p)\leq C\cdot\deg(p) \ \ \ \forall p\in \k[x]\backslash I$$
then $I$ is generated by algebraic power series.

\end{theorem}
\vspace{0.2cm}

\begin{rmk}
We remark that  the hypothesis of Theorem \ref{main2} are satisfied for a principal ideal $I$. In particular Theorems \ref{main} and \ref{main2} provide a criterion for a principal ideal to be generated by an algebraic power series.
\end{rmk}

\begin{rmk}
We will see in Section \ref{contex} that Theorem \ref{main2} is not true in general.
\end{rmk}

These two results  are generalizations  of previous results of S. Izumi (see \cite{Iz92}, \cite{Iz2}, \cite{Iz98} where he proved Corollary \ref{izu} when $\cha(\k)=0$, $s=1$ and $N$ is a prime ideal of $\k\lb x\rb$) and Theorem \ref{main2} when $I$ is prime and $\cha(\k)=0$. \\
  The proof of Theorem \ref{main2}  uses Hilbert-Samuel functions and is inspired by the proof given in   \cite{Iz2}. The proof of Theorem \ref{main} is more difficult and is the main subject of this paper. In fact the first difficulty occurs already when $s=1$ and $N$ is an ideal of $\k\lb x\rb$ which is not prime. Corollary \ref{izu} in the case of a prime ideal  has been proven by S. Izumi in \cite{Iz92} in the complex analytic case using resolution of singularities of Moishezon spaces and then for any field of characteristic zero using basic field theory in \cite{Iz98}. But when $N$ is not prime his proof does not adapt at all and the general case cannot be reduced to the case proven by S. Izumi. \\
 \\
 The proof of Theorem \ref{main} that we give here is done by induction on $s$ and $n$. The induction steps require  two effective division results in the rings of algebraic power series which may be of general interest. These are the following ones:\\
 \begin{itemize}
\item[i)] In the case of the Weierstrass Division of an algebraic power series $f$ by another algebraic power series it is  proven by J.-P. Lafon that the remainder and the quotient of the division are algebraic power series \cite{La}. The problem solved here is to bound the complexity of the division, i.e. bound the complexity of the quotient and the remainder of the division in function of the complexity of the input data. This is Theorem \ref{division_sep} and is the main tool to solve the  next division problem.
Let us mention that this problem is partially solved in \cite{As} Section 4 - see Theorem 4.6.\\
\item[ii)] Bounding the complexity of the Ideal Membership Problem in the ring of algebraic power series, i.e. if an algebraic power series $f$ is in the ideal generated by algebraic power series $g_1$, $\cdots$, $g_p$, bound the complexity of algebraic power series $a_1$, $\cdots$, $a_p$ such that 
$$f=a_1g_1+\cdots+a_pg_p.$$
This is  Theorem \ref{IMP_alg}.\\
\end{itemize}
The complexity invariants associated to an algebraic power series $f$ are its degree and its height. The first one is  the  degree of the field extension $\k(x)\lgw \k(x,f)$ and the second one has been defined above. In particular we will prove that the previous complexity
 problems admit a solution which is linear with respect to the height of $f$ (but it is not linear which respect to the other data). This is exactly what we need to prove Theorem \ref{main}.\\
 \\
Finally we apply our main theorem to give a partial answer to a question of H. Hironaka. When $f$, $g_1$, $\cdots$, $g_s$ are formal power series, we can write
$$f=a_1g_1+\cdots+a_sg_s+r$$ where the non-zero monomials in the expansion of $r$ are not divisible by the initial terms of the $g_i$ (see Section \ref{GHG} for precise definitions). When the power series $f$ and the $g_i$ are convergent then $r$ is also convergent. This result has been proven by H. Grauert in order to study versal deformations of isolated singularities of analytic hypersurfaces \cite{Gr} and  then by H. Hironaka to study resolution of singularities \cite{Hi64}. But when $f$ and the $g_i$ are algebraic power series, then $r$ is not an algebraic power series in general and H. Hironaka raised the problem of characterizing such power series $r$ (see \cite{Hir}). In this case we prove that such power series $r$ are not too transcendental (see Theorem \ref{remainder_division}). More precisely if we write $r$ as $\displaystyle r=\sum_{k=0}^{\infty}r_{n(k)}$ where $r_{n(k)}$ is a non-zero homogeneous polynomial of degree $n(k)$ and the sequence $(n(k))_k$ is strictly increasing, we show that
$$\limsup_{k\lgw \infty}\frac{n(k+1)}{n(k)}<\infty.$$
Let us mention that this division problem appears also in combinatorics: the generating series of walks confined in the first quadrant are solutions of such a division but are nor algebraic nor $D$-finite in general  (see \cite{H-K} or \cite{KK}).\\
\\
Let us mention that  the kind of estimates given in Corollary \ref{izu}, i.e. estimates of the form
$$\ord_Mp\leq \g(\deg(p))$$ where $\g:\N\lgw \N$ is an increasing function, $s=1$ and $N$ is an ideal of analytic functions are called \emph{zero estimates} in the literature. Finding such estimates for particular classes of functions is an important subject of research in transcendence theory, in particular when the ideal $N$ is generated by analytic functions  of the form
$$x_k-f_k(x_1,\cdots,x_{k-1}), \cdots, x_n-f_n(x_1,\cdots, x_{k-1})$$
for some $k<n$ and $f_k$,  $\cdots$, $f_n$  solutions of differential equations (see \cite{Sh}, \cite{BB85}, \cite{N} for instance) or functional equations ($q$-difference equations or Mahler functions - see \cite{Ni90} for instance).\\
\\
We should also mention that the complexity of the Weierstrass Division for restricted power series defined over the ring of $p$-adic integers which are algebraic over $\Q[x]$ has been solved in \cite{As}. The complexity of the Ideal Membership Problem is also solved in this situation. In this case the definition of the height of an algebraic power series is more complicated.\\
\\
The paper is organized as follows: after giving the list of notations used in the paper in Section \ref{Not}, we define the height of an algebraic power series in Section \ref{h} and give the first properties of it. In Section \ref{W} we prove an effective Weierstrass Division Theorem (see Theorem \ref{division_sep}). In Section \ref{S_IMP} we give some results about the Ideal Membership Problem in rings which are localizations of  rings of polynomials (see Theorem \ref{IMPlocal} and Proposition \ref{IMP_local2}) and in Section \ref{S_IMP_alg} we give an effective Ideal Membership theorem for algebraic power series rings (see Theorem \ref{IMP_alg}). Then Section \ref{proof_main} is devoted to the proof of Theorem \ref{main} and Section   \ref{proof_main2} to the proof of Theorem \ref{main2}. In Section \ref{contex} is given an example showing that the hypothesis of Theorem \ref{main2} cannot being relaxed. The next three sections concern the Grauert-Hironaka-Galligo Division Theorem: in Section \ref{GHG} we state this theorem and give the example of Gabber-Kashiwara showing that the remainder of such division of an algebraic power series by another one is not algebraic in general. We show in Section \ref{KG_gen} that the example of Gabber-Kashiwara is generic in some sense, i.e. in general the division of an algebraic power series by another one does not have an algebraic remainder (see Proposition \ref{prop2}). Finally we prove in Section \ref{Div_alg} our gap theorem for remainders of such division (see Theorem \ref{remainder_division}).

\begin{rmk}
We show in Example \ref{ex2} that the bound in Corollary \ref{izu} is sharp. For Theorem \ref{main} it is not clear if   such bound is sharp. Indeed, let $f$ be an algebraic power series and $M=\k\lb x\rb/I$ where $I$ is an ideal generated by algebraic power series. Let 
$$a_d(x)T^d+a_{d-1}(x)T^{d-1}+\cdots+a_0(x)$$
be the minimal polynomial of $f$. Then we have
$$(a_df^{d-1}+a_{d-1}f^{d-2}+\cdots+a_1)f=-a_0.$$
We set $g:=a_df^{d-1}+a_{d-1}f^{d-2}+\cdots+a_1$. If $a_0\notin I$, then 
$$\ord_{M}(f)\leq \ord_M (gf)=\ord_M(a_0)\leq C\, \h(f)$$
where $C$ is the constant of Corollary \ref{izu} since $a_0(x)$ is a polynomial of degree $\leq \h(f)$. This shows that in general the function $C$ of Theorem \ref{main} can be chosen to be independent of $\Deg(f)$ except maybe when $a_0(x)\equiv 0$ in $M$.\\
\end{rmk}

\noindent
\textbf{Acknowledgment. }
The author would like to thank Paco Castro-Jim\'enez and Herwig Hauser for the discussions they had about the Weierstrass division Theorem in the algebraic case. He would like also thank Matthias Aschenbrenner for communicating the reference  \cite{As}.\\
The author is really grateful to  the referee for their helpful suggestions for improving the readability of the article.

\section{Notations}\label{Not}
In the whole paper $\k$ denotes a field of any characteristic. Let $n$ be a  non-negative integer and set
$$x:=(x_1,\cdots,x_n)\text{ and } x':=(x_1,\cdots,x_{n-1}).$$
The ring of polynomials in $n$ variables over $\k$ will be denoted by $\k[x]$ and its field of fractions by $\k(x)$. The ring of formal power series in $n$ variables over $\k$ is denoted by $\k\lb x\rb$ and its field of fractions by $\k((x))$. An algebraic power series is a power series $f(x)\in\k\lb x\rb$ such that
$$P(x,f(x))=0$$ for some non-zero polynomial $P(x,T)\in\k[x,T]$ where $T$ is a single indeterminate. The set of algebraic power series is a local subring of $\k \lb x\rb$ denoted by $\k\lag x\rag$.\\
When $\k$ is a valued field we denote by $\k\{ x\}$ the ring of convergent power series in $n$ variables over $\k$. We have
$$\k[x]\subset \k\lag x\rag\subset \k\{ x\}\subset \k\lb x\rb.$$
We will denote by $\K$ an algebraic closure of $\k((x'))=\k((x_1,\cdots,x_{n-1}))$.\\
\\
For a polynomial $p\in\k[x]$ we denote by $\deg(p)$ its total degree with respect to the variables $x_1$, $\cdots$, $x_n$. If $y:=(y_1,\cdots,y_m)$ is a new set of indeterminates and $p\in\k[x,y]$ we denote by 
$$\deg_{(y_1,\cdots,y_m)}(p)$$
the degree of $p$ seen as a polynomial in $\mathbb{K}[y]$ where $\mathbb{K}:=\k(x)$. When $p\in\k[x]^s$ for some $s$, we denote by $\deg(p)$ the  maximum of the degrees of the components of $p$.\\
\\
For an algebraic power series $f\in\k\lag x\rag$, the height of $f$ is the maximum of the degrees of the coefficients of the minimal polynomial of $f$ (see Definition \ref{height}). The height of a vector of algebraic power series is the maximum of the heights of its components.\\
\\	
When $(A,\m)$ is a local ring we set 
$$\ord_A(x):=\sup\{k\in \N\ / \ x\in\m^k\}\in\N\cup\{\infty\} \ \ \ \forall x\in A.$$
If $M$ is a finite $A$-module we set 
$$\ord_M(m):=\sup\{k\in\N\ / \ m\in \m^kM\}\ \ \ \  \forall m\in M.$$
When $A=\k\lb x\rb$ we write $\ord$ instead of $\ord_{\k\lb x\rb}$. For an ideal of $\k\lb x\rb$ generated by $g_1$, $\cdots$, $g_p$ we define
$$\ord_{g_1,\cdots,g_p}(f):=\sup\{k\in\N \ / \ f\in (g_1,\cdots,g_p)^k\}\in\N\cup\{\infty\}.$$\\


\section{Height and degree of algebraic power series}\label{h}

\begin{definition}
Let $\a$ be an element of an algebraic closure of $\k(x)$ (for example an algebraic power series).
The morphism $\phi:\k[x,T]\lgw \k(x,\a)$ defined by sending every polynomial  $P(x,T)$ onto $P(x,\a)$ is not injective and its kernel is a  prime ideal $\p$ of $\k[x,T]$. If $\het(\p)\geq 2$ then $\p\cap \k[x]\neq (0)$ and there would exist a non-zero polynomial $P(x)\in\k[x]$ whose image by $\phi$ is zero which is not possible. Thus $\het(\p)=1$ and $\p$  is a principal ideal. If $P(x,T)$ is a generator of $\p$ then any other generator of this ideal is equal to $P(x,T)$ times a non-zero element of $\k$. Such a generator is called a minimal polynomial of $\a$. By abuse of language we will often refer to such an element by \emph{the} minimal polynomial of $\a$.
\end{definition}

\begin{definition}\cite{AB}\label{height}
Let $P(x,T)\in\k[x,T]$. The height of $P$ is the maximum of the degrees of the coefficients of $P(x,T)$ seen as a polynomial in $T$.\\
Let $\a$ be an algebraic element over $\k(x)$. The height of $\a$ is the height of its minimal polynomial and is denoted by $\h(\a)$. Its degree is the degree of its minimal polynomial or, equivalently, the degree of the field extension $\k(x)\lgw \k(x,\a)$ and is denoted by $\Deg(\a)$.\\
When $\a=(\a_1,\cdots,\a_m)$ is a vector of algebraic elements over $\k(x)$ the height of $\a$, $\h(\a)$, is the maximum of the heights of the components of $\a$ and the degree of $\a$, $\Deg(\a)$, is the degree of the field extension $\k(x)\lgw \k(x,\a_1,\cdots,\a_m)$
\end{definition}

\begin{rmk}
If $P(x,T)\in\k[x,T]$ is the minimal polynomial of an algebraic element $\a$, then $\h(\a)=\deg_x(P)$ and $\Deg(\a)=\deg_T(P)$. In particular for $Q(x,T)\in\k[x,T]$ with $Q(\a)=0$, $P$ divides $Q$ hence we have $\h(\a)\leq \deg_x(Q)$ and $\Deg(\a)\leq \deg_T(Q)$.
\end{rmk}

\begin{ex}
Let $f$ be a polynomial in $\k[x]$, then $\h(f)=\deg(f)$ and $\Deg(f)=1$ since the minimal polynomial of $f$ is $T-f$.\\
Let $f/g$ be a rational function in $\k(x)$. Then $\h(f/g)=\max\{\deg(f),\deg(g)\}$ and $\Deg(f/g)=1$ since the minimal polynomial of $f/g$ is $gT-f$.\\
If $\a$ is algebraic over $\k(x)$, then $1/\a$ also and $\h(1/\a)=\h(\a)$ and $\Deg(1/\a)=\Deg(\a)$.\\
If $f(x)$ is an algebraic power series and $M\in\text{Gl}_n(\k)$ then $f(Mx)$ is also algebraic and $\h(f(Mx))=\h(f(x))$ and $\Deg(f(Mx))=\Deg(f(x)).$
\end{ex}

\begin{rmk}
There exists another measure of the complexity of an algebraic element $\a$ over $\k(x)$ (and so, in particular, of an algebraic power series). This one is defined to be the total degree of the minimal polynomial of $\a$ and denoted by $\co(\a)$ (cf. \cite{Ra} or \cite{AMR}). Thus we have
$$\frac{\h(\a)+\Deg(\a)}{2}\leq\max\{\h(\a),\Deg(\a)\}\leq\co(\a)\leq \h(\a)+\Deg(\a).$$
This shows that $\co(\a)$ is equivalent to $\h(\a)+\Deg(\a)$.
Moreover these bounds are sharp. Indeed let $P_n(T):=(1+x^n)T^n-1$ (where $x$ is a single variable and $n\in\N$ is not a multiple of the characteristic of $\k$). Then $P_n(T)$ is irreducible and has a root $f_n$ in $\k\lag x\rag$. Thus $\h(f_n)=\Deg(f_n)=n$ and $\co(f_n)=2n$. On the other hand the polynomial $Q_n(T):=T^n-(1+x^n)$ is irreducible and has a root $g_n$ in $\k\lag x\rag$. Thus $\h(g_n)=\Deg(g_n)=\co(g_n)=n$.\\
 For an algebraic power series $f$ we choose to use $\h(f)$ instead of $\co(f)$ since the complexity of the Weierstrass Division Theorem is linear in $\h(f)$ but not in $\co(f)$ (it is not linear in $\Deg(f)$ - see Theorem \ref{division_sep}). Indeed we need to prove the existence of a bound in Theorem \ref{main} which is linear in $\h(f)$.\\
\end{rmk}

\begin{lemma}(\cite{AB} Lemma 4.1)\label{inequalities_height}
Let $\a_1$, $\cdots$, $\a_p$ be algebraic elements over $\k(x)$ and $a_1$, $\cdots$, $a_p\in\k(x)$. Then we have:

\begin{equation}\label{1}\tag{i}\Deg(a_1\a_1+\cdots+a_p\a_p)\leq \Deg(\a_1)\cdots\Deg(\a_p),\end{equation}
\begin{equation}\label{2}\tag{ii}\h(a_1\a_1\cdots+a_p\a_p)\leq p\cdot\Deg(\a_1)\cdots\Deg(\a_p)(\max_i\{\h(\a_i)\}+\max_j\{\h(a_j)\}),\end{equation}
\begin{equation}\label{3}\tag{iii}\h(a_1+\a_1)\leq  \h(\a_1)+\Deg(\a_1)\cdot\h(a_1),\end{equation}
\begin{equation}\label{6}\tag{iv} \h(a_1\a_1)\leq \h(\a_1)+\Deg(\a_1)\cdot\h(a_1) \end{equation} 
\begin{equation}\label{4}\tag{v}\Deg(\a_1\cdots \a_p)\leq \Deg(\a_1)\cdots\Deg(\a_p),\end{equation}
\begin{equation}\label{5}\tag{vi}\h(\a_1\cdots \a_p)\leq p\cdot\Deg(\a_1)\cdots\Deg(\a_p)\max_i\{\h(\a_i)\}.\end{equation}\\
\end{lemma}

\begin{proof}
All these inequalities are proven in \cite{AB} except the third  and the fourth ones that we prove here. Let us begin with the third one:\\
 Let $P(x,T)$ be the minimal polynomial of $\a_1$ and let us write $a_1(x)=b(x)/c(x)$ for some polynomials $b(x)$ and $c(x)$. Then 
$$Q(x,T):=c(x)^{\deg_T(P)}P(x,T-a_1)$$ is a polynomial vanishing at $\a_1+a_1$. Thus
$$\h(\a_1+a_1)\leq \deg_x(Q(x,T))\leq \h(\a_1)+\Deg(\a_1)\h(a_1)$$
since $\Deg(\a_1)=\deg_T(P)$ and $\h(a_1)\geq \max\{\deg(b(x)),\deg(c(x))\}$.\\
To prove Inequality \eqref{6} let $P(x,T)$, $b$, $c$ as above. Then $b^{\deg_T(P)}P(x,c/bT)$ is a polynomial vanishing at $a_1\a_1$. So
$$\h(a_1\a_1)\leq \Deg(\a_1)\cdot\h(a_1)+\h(\a_1).$$
\end{proof}

\begin{lemma}\label{height_ord}
  For an algebraic power series $f$ we have:
   $$ \ord(f)\leq \h(f).$$
   Moreover for any integer $1\leq i\leq n$ we have:
   $$\h(f(0,\cdots,0,x_i,\cdots,x_n))\leq \h(f)$$
   and 
   $$\ord_{x_i,\cdots,x_n}(f(0,\cdots,0,x_i,\cdots,x_n))\leq \h(f).$$
\end{lemma}
\vspace{0.2cm}
\begin{proof} Let $P(T)=a_dT^d+\cdots+a_1T+a_0$ be the minimal polynomial of $f$. Since $P(f)=0$ there are two integers $0\leq i<j\leq d$ such that $\ord(a_if^i)=\ord(a_jf^j)$. Thus 
  $$\ord(f)=\frac{\ord(a_i)-\ord(a_j)}{j-i}\leq \ord(a_i)\leq \deg(a_i).$$
  This proves the first inequality. The second one is proven by noticing that if $P(x_1,\cdots, x_n,f(x_1,\cdots,x_n))=0$, then $P(0,x_2,\cdots,x_n, f(0,x_2,\cdots,x_n))=0$. Since $P$ is the minimal polynomial of $f$, then $P$ is not divisible by $x_1$, thus $$P(0,x_2,\cdots,x_n,T)\neq 0.$$ This proves that  $f(0,x_2,\cdots,x_n)$ is an algebraic power series and its minimal polynomial divides $P(0,x_2,\cdots,x_n,T)$, hence  
  $$\h(f(0,x_2,\cdots,x_n))\leq \h(f).$$ 
  The first inequality implies
  $$\ord_{x_2,\cdots,x_n}(f(0,x_2,\cdots,x_n))\leq \h(f).$$
  Hence the last two inequalities are proven by induction on $i$.
  \end{proof}
  \vspace{0.2cm}
  
  \begin{rmk}
  A formal power series $f$ is said to be $x_n$-regular if $f(0, \cdots,0,x_n)\neq 0$. In this case we say that $f$ is $x_n$-regular   of order $d$ if $f(0, \cdots,0,x_n)$ is a power series of $\k\lb x_n\rb$ of order $d$.\\
  By the previous lemma, if an algebraic power series $f$ is $x_n$-regular of order $d$ then $d\leq \h(f)$. 
  \end{rmk}
 \begin{rmkdef} \label{roots} 
  Let $\K$ be an algebraic closure of $\k((x'))$ where $x':=(x_1,\cdots,x_{n-1})$. The $(x')$-valuation $\ord_{x'}$ defined on $\k((x'))$ extends uniquely to $\K$ and is still denoted by $\ord_{x'}$. The completion of $\K$ for the valuation $\ord_{x'}$  is denoted by $\wdh{\KK}_{n-1}$. Let $\a\in \wdh{\KK}_{n-1}$ such that $\ord_{x'}(\a)>0$ and $f$ be a formal power series. Then $f(x',\a)$ is well defined in $\wdh{\KK}_{n-1}$. If $f(x',\a)=0$ we call $\a$ a \emph{root} of $f$. \\
  \\
   If $f$ is an algebraic power series, $P(x,T)$ is the minimal polynomial of $f$ and $\a$ is a root of $f$ then $P(x',\a,0)=0$ thus $\a$ is algebraic over $\k(x')$.\\
  \\
  Let $f$ be a formal power series which is $x_n$-regular of order $d$. Then, by the Weierstrass Preparation Theorem, there exist a unit $v$ and a Weierstrass polynomial $P=x_n^d+a_1(x')x_n^{d-1}+\cdots +a_d(x')$ such that $f=vP$. The polynomial $P$ is called the \emph{Weierstrass polynomial of $f$}.
  Let $\a\in \K$ be a root of $P$. Since $\ord_{x'}(a_i(x'))>0$ for any $i$ we have $\ord_{x'}(\a)>0$. Thus $f(x',\a)$ and $v(x',\a)$ are  well defined in $\wdh{\KK}_{n-1}$ and $f(x',\a)=0$. On the other hand if $\a\in\wdh {\KK}_{n-1}$ is a root of $f$, since $\ord_{x'}(\a)>0$ then $v(x',\a)\neq 0$ in $\wdh{\KK}_{n-1}$, thus  $P(x',\a)=0$. In particular $\a$ is a root of the polynomial $P$ in the usual sense thus $\a\in\K$.\\
   This proves that the roots of $f$ are exactly the roots (in the usual sense) of $P$ seen as a polynomial in $x_n$ and are elements of $\K$.\\
  \end{rmkdef}

   \begin{lemma}\label{height_lemma}
  Let $\a\in\K$ be a root of a $x_n$-regular algebraic power series $f$. Then $\a$ is algebraic over $\k(x)$ and
  $$\h(\a)\leq \h(f)\ \text{ 
  and }\ 
 \Deg(\a)\leq \h(f).$$
 Moreover if $\a_1$, $\cdots$, $\a_d$ are distinct roots of $f$, then
 $$[\k(x',\a_1,\cdots,\a_d):\k(x')]\leq \h(f)!$$\\
  \end{lemma}

\begin{proof}
Let $P(x,T)$ be the minimal polynomial of $f$. Since $f(x',\a)=0$ we have
$$P(x',\a,0)=0.$$
Thus $P(x',T,0)$ is a non-zero polynomial vanishing at $\a$, proving that $\a$ is algebraic, and
$$\h(\a)\leq \deg_{x'}(P(x',T,0))\leq \deg_{(x',x_n)}(P(x',x_n,T))=\h(f)$$
and
$$\Deg(\a)\leq \deg_T(P(x',T,0))\leq \deg_{x_n}(P(x',x_n,T))\leq \h(f).$$\\
Moreover $P(x',T,0)$ is a polynomial having $\a_1$, $\cdots$, $\a_d$ as roots. Thus a splitting field of $P(x',T,0)$ over $\k(x')$ contains these roots, thus 
 $$[\k(x',\a_1,\cdots,\a_d):\k(x')]\leq \deg_T(P(x',T,0))!$$
\end{proof}

\begin{lemma}\label{height_compo1}
Let $\a$ be algebraic over $\k(x)$ with $\ord_{x}(\a)>0$. Let $g(x,y)$ be an algebraic power series where $y$ is a single variable. Then $g(x,\a)$ is algebraic over $\k(x)$ and 
$$\h(g(x,\a))\leq \h(g)\cdot(\h(\a)+\Deg(\a))$$
$$\Deg(g(x,\a))\leq \Deg(\a)\cdot\Deg(g).$$\\
\end{lemma}

\begin{proof}
Let $P(x,y,T)\in\k[x,y,T]$ be the minimal polynomial of $g$ and $Q(x,T)\in\k[x,T]$ be the minimal polynomial of $\a$. 
Then
$$P(x,\a,g(x,\a))=0$$
and $P(x,\a,T)\neq0$ otherwise $P(x,y,T)$ is divisible by $Q(x,y)$ which is impossible since $P$ is assumed to be irreducible. Thus $g(x,\a)$ is algebraic over $\k(x,\a)$, hence over $\k(x)$. If we denote by  $R(x,T)$  the resultant of $P(x,y,T)$ and $Q(x,y)$ seen as polynomials in $y$:
$$R(x,T):=\Res_y(P(x,y,T),Q(x,y))\neq 0,$$
then $R(x,T)$
is a polynomial of $\k[x][T]$ vanishing at $g(x,\a)$. Let us write
$$P(x,y,T)=a_0(x,T)+a_1(x,T)y+\cdots+a_h(x,T)y^h \text{ with } a_h\neq 0.$$
Moreover, since $P(x,y,T)$ is the minimal polynomial of $g$ (as a polynomial in $T$), for all $i$ we have 
$$\deg_x(a_i)+i\leq \h(g),$$ 
$$\deg(a_i)\leq \h(g)+\Deg(g)-i\leq \h(g)+\Deg(g),$$ 
$$\deg_T(a_i)\leq \Deg(g).$$ 
In particular $h\leq \h(g)$ and $\deg_{x}(a_i)\leq \h(g)$ for all $i$. We write

$$Q(x,y)=b_0(x)+b_1(x)y+\cdots+b_e(x)y^e$$ with $e=\Deg(\a)$ and $\deg(b_i)\leq \h(\a)$ for all $i$.
Since $R(x,T)$ is homogeneous of degree $h$ in $b_0$, $\cdots$, $b_e$ and homogeneous of degree $e$ in $a_0$, $\cdots$, $a_h$, we see that 

$$\deg_T(R(x,T))\leq e\cdot \Deg(g)= \Deg(\a)\cdot\Deg(g)$$
and 
$$\h(R(x,T))\leq h\cdot \h(\a)+e\cdot\h(g). $$
 This proves the lemma.
\end{proof}
\vspace{0.2cm}

\begin{corollary}\label{aschen}
Let $f$ be a $x_n$-regular algebraic power series and let $\a_1$, $\cdots$, $\a_d$ be distinct roots of $f$ in $\K$. Let $g\in\k\lag x\rag$ be any algebraic power series. Then
$$\left[\k(x',\a_1,\cdots,\a_d,g(x',\a_1),\cdots,g(x',\a_d)):\k(x')\right]\leq \h(f)!\Deg(g)^d.$$
\end{corollary}
\vspace{0.2cm}
\begin{proof}
By Lemmas \ref{height_lemma} and \ref{height_compo1} the degree of this field extension is finite. By the proof of Lemma \ref{height_compo1} we have, for any $i$:
$$[\k(x',\a_1,\cdots,\a_d,g(x',\a_i):\k(x',\a_1,\cdots,\a_d)]\leq \Deg(g).$$
Thus 
$$[\k(x',\a_1,\cdots,\a_d,g(x',\a_1),\cdots,g(x',\a_d)):\k(x',\a_1,\cdots,\a_d)]\leq \Deg(g)^d.$$
Hence the result follows by Lemma \ref{height_lemma}.
\end{proof}
\vspace{0.2cm}
\begin{rmk}
Let $g(x,y)$ be an algebraic power series where $y=(y_1,\cdots, y_m)$ is a tuple of indeterminates and let $a_1(x)$, $\cdots$, $a_m(x)$ be algebraic power series vanishing at 0.  If $P(x,y,T)$ is the minimal polynomial of $g$, then 
$$P(x,a(x),g(x,a(x))=0$$
but it may happen that
$$P(x,a(x),T)=0.$$
Hence the previous proof does not extend directly to this case.
For example let
$$P_1(x,y_1):=y_1^2-(1+x)$$
$$P_2(x,y_2):=y_2^2-(1+x)$$
where $x$ is a single variable and $\k$ is a field of characteristic $\neq 2$ in which $-7$ is a square. Then $P_1$ and $P_2$ have a common root in $\k\lag x\rag$, say $a(x)$. Let
$$P(x,y,T):=(P_1+P_2)T^2+P_1T+P_2.$$
The discriminant of $P$ is equal to
$$\D:=P_1^2-4P_2(P_1+P_2)=$$
$$=y_1^4-2(1+x)y_1^2+(1+x)^2+4(y_2^2-(1+x))(2(1+x)-y_1^2-y_2^2)=$$
$$=y_1^4+2(1+x-2y_2^2)y_1^2+12(1+x)y_2^2-4y_2^4-7(1+x)^2$$
and is not a square in $\k[x,y_1,y_2]$, thus $P$ is irreducible in $\k[x,y,T]$. But $\D$ is unit in $\k\lag x,y\rag$ since $\D(0,0,0)=-7$, and $P_1+P_2$ is also a unit. So $\D$ has a root square in $\k\lag x,y\rag$ since $\cha(\k)\neq 2$ and $-7$ is a square in $\k$. Thus in this case $P(x,y,T)$ has two distinct roots in $\k\lag x,y\rag$. But here
$$P(x,a(x),a(x),T)=0.$$
\end{rmk}
\vspace{0.2cm}
\noindent Nevertheless we can extend Lemma \ref{height_compo1} as follows:
\vspace{0.2cm}

\begin{lemma}\label{height_compo2}
Let $g(x,y)$ be an algebraic power series where $y=(y_1,\cdots, y_m)$ is a tuple of indeterminates and let $a_1(x)$, $\cdots$, $a_m(x)$ be algebraic power series vanishing at 0. Then 
$$\h(g(x,a(x)))\leq \left(\prod_{i=1}^m(\h(a_i)+\Deg(a_i))\right)\cdot\h(g),$$
$$\Deg(g(x,a(x)))\leq \left(\prod_{i=1}^m\Deg(a_i)\right)\cdot\Deg(g).$$
\end{lemma}

\begin{proof}
Let us set
$$g_0(x,y_1,\cdots,y_m):=g(x,y),$$
$$g_1(x,y_2,\cdots,y_{m}):=g_0(x,a_1(x),y_2,\cdots,y_{m}),$$
$$g_2(x,y_3,\cdots,y_{m}):=g_1(x,a_2(x),y_3,\cdots,y_{m}),$$
$$\cdots\cdots\cdots$$
$$g_m(x)=g_{m-1}(x,a_m(x))=g(x,a(x)).$$
Then by Lemma \ref{height_compo1}, we have
$$\Deg(g_i)\leq \Deg(a_i)\cdot\Deg(g_{i-1}),$$
$$\h(g_i)\leq \h(g_{i-1})(\h(a_i)+\Deg(a_i)).$$
This proves the lemma.
\end{proof}
\vspace{0.2cm}

\begin{lemma}\label{height_derivation}
Let $f$ be an algebraic power series. Then $\frac{\partial f}{\partial x_n}$ is an algebraic power series and
$$\h\left(\frac{\partial f}{\partial x_n}\right)\leq 4\Deg(f)^{2\Deg(f)+4}\h(f), $$
$$\Deg\left(\frac{\partial f}{\partial x_n}\right)\leq \Deg(f) .$$
\end{lemma}

\begin{proof}
Let $P(x,T)$ be the minimal polynomial of $f$. Since $P(x,f)=0$ we have
$$\frac{\partial P}{\partial x_n}(x,f(x))+\frac{\partial f}{\partial x_n}(x)\frac{\partial P}{\partial T}(x,f(x))=0.$$
Since $f$ is separable over $\k(x)$ (indeed $\k\langle x\rangle$ is the Henselization of $\k[x]_{(x)}$ and the morphism from a local ring to its Henselization is always a separable morphism - see \cite{Na} p. 180), then $\frac{\partial P}{\partial T}\neq 0$. Moreover  $P$ is the minimal polynomial of $f$ so $\frac{\partial P}{\partial T}(x,f(x))\neq 0$. Thus $\frac{\partial f}{\partial x_n}(x)$ is an algebraic power series and 
$$\frac{\partial f}{\partial x_n}(x)=-\frac{\partial P}{\partial x_n}(x,f(x))/\frac{\partial P}{\partial T}(x,f(x))\in \k(x,f).$$
So we obtain 
$$\Deg\left(\frac{\partial f}{\partial x_n}(x)\right)\leq \Deg(f)$$
and, by Lemma \ref{inequalities_height} \eqref{5},
\begin{equation}\label{der}\h\left(\frac{\partial f}{\partial x_n}(x)\right)\leq 2\Deg(f)^2\max\left\{\h\left(\frac{\partial P}{\partial x_n}(x,f(x))\right), \h\left(\frac{\partial P}{\partial T}(x,f(x))\right)\right\}.\end{equation}
We have 
$$\frac{\partial P}{\partial T}(x,f(x))=\sum_{i=0}^{\Deg(f)-1}a_i(x)f(x)^i$$

for some polynomials $a_i(x)$ with $\deg(a_i)\leq \h(f)$. Thus, by Lemma \ref{inequalities_height} \eqref{2},

$$\h\left(\frac{\partial P}{\partial T}(x,f(x))\right)\leq \Deg(f)\cdot\Deg(f^0)\cdots\Deg(f^{\Deg(f)-1})(\max_j\{\h(f^j)+\h(a_j)\})$$

$$\leq \Deg(f)^{\Deg(f)}((\Deg(f)-1)\Deg(f)^{\Deg(f)-1}\h(f)+\h(f))\leq \Deg(f)^{2\Deg(f)}\h(f)$$
since $f^0=1$, $f^i \in\k(x,f)$ for all $i$ and 
$$\h(a_i)=\deg(a_i)\leq \h(f),\ \  \h(f^i)\leq i\Deg(f)^i\h(f)\ \ \  \forall i$$ by Lemma \ref{inequalities_height} \eqref{5}.\\
\\
We also have 
$$\frac{\partial P}{\partial x_n}(x,f(x))=\sum_{i=0}^{\Deg(f)}b_i(x)f(x)^i$$
for some polynomials $b_i(x)$ with $\deg(b_i)\leq \h(f)$. Thus in the same way

$$\h\left(\frac{\partial P}{\partial x_n}(x,f(x))\right)\leq  \Deg(f)\cdot\Deg(f^0)\cdots\Deg(f^{\Deg(f)})(\max\{\h(f^i)+\h(f)\})$$

$$\leq \Deg(f)^{\Deg(f)+1}(\Deg(f)\Deg(f)^{\Deg(f)}\h(f)+\h(f))\leq 2\Deg(f)^{2\Deg(f)+2}\h(f).$$\\
Replacing these inequalities in Inequality \eqref{der} we are done.
\end{proof}
\vspace{0.2cm}

\begin{lemma}\label{height_power}
Let $f(x,y)$ be an algebraic power series where $x=(x_1,\cdots,x_n)$ and $y$ is a single variable. Let $q$ be a positive integer. Then $f(x,y^q)$ is an algebraic power series with the same degree as $f(x,y)$ and
$$\h(f(x,y))\leq \h(f(x,y^q)) \leq q\h(f(x,y)).$$
\end{lemma}

\begin{proof}
If $P(x,y,T)$ is the minimal polynomial of $f(x,y)$, then $P(x,y^q,T)$ is a polynomial having $f(x,y^q)$ as a root. Thus $f(x,y^q)$ is an algebraic power series.\\
Since $\k[x,y,T]$ is a free $\k[x,y^q,T]$-module with basis $1$, $y$, $\cdots$, $y^{q-1}$, if $Q(x,y,T)$ is the minimal polynomial of $f(x,y^q)$, we can write in a unique way
$$Q(x,y,T)=Q_0(x,y^q,T)+Q_1(x,y^q,T)y+\cdots+Q_{q-1}(x,y^q,T)y^{q-1}$$
where the $Q_i(x,y^q,T)$ are polynomials. Since $Q(x,y,f(x,y^q))=0$, then we see that $Q_i(x,y^q,f(x,y^q))=0$ for all $i$. Since $Q$ is the minimal polynomial of $f(x,y^q)$, then $Q$ divides all the $Q_i(x,y^q,T)$, hence $Q=Q_0$ and $Q_i=0$ for all $i>0$. This shows that the minimal polynomial of $f(x,y^q)$ has coefficients in $\k[x,y^q]$. \\
Now if $Q(x,y^q,T)$ is the minimal polynomial of $f(x,y^q)$ then $Q(x,y,f(x,y))=0$. This proves that $P(x,y,T)$ is the minimal polynomial of $f(x,y)$ if and only if $P(x,y^q,T)$ is the minimal polynomial of $f(x,y^q)$.\\
\\
Since 
$$\deg_T(P(x,y,T))=\deg_T(P(x,y^q,T))$$
we see that $f(x,y)$ and $f(x,y^q)$ have the same degree.\\
Moreover 
$$\deg_{(x,y)}(P(x,y,T))\leq \deg_{(x,y)}(P(x,y^q,T)) \leq q\cdot\deg_{(x,y)}(P(x,y,T)).$$
This shows the inequalities concerning the heights.
\end{proof}

\vspace{0.2cm}

\begin{lemma}\label{height_extraction}
Let $f(x,y)$ be an algebraic power series where $y$ is a single variable and $q$ be a positive integer.  Let us write $q=rp^e$ where $p=\cha(\k)$, $e\in\N$ and $\gcd(r,p)=1$ (we set $e=0$ when $\cha(\k)=0$ and by convention $q=r$). Let us write
$$f(x,y)=f_0(x,y^q)+f_1(x,y^q)y+\cdots+f_{q-1}(x,y^q)y^{q-1}.$$
Then the power series $f_i(x,y^q)$ are algebraic and for any $0\leq i\leq q-1$ we have
$$\h(f_i(x,y^q))\leq q^2p^{\frac{e(e+1)}{2}}4^q\Deg(f)^{2q\Deg(f)+5q}\left(\h(f)+\frac{q(q-1)}{2}\right) \text{ if } e>0,$$
$$\h(f_i(x,y^q))\leq \Deg(f)^q(q\h(f)+q-1) \text{ if } e=0,$$
$$\Deg(f_i(x,y^q))\leq  \Deg(f)^r.$$
\end{lemma}
\vspace{0.2cm}
\begin{proof}
We need to consider several cases:\\
\\
(1) First we assume that $e=0$ i.e. $\gcd(q,p)=1$. By taking a finite extension of $\k$ we may assume that $\k$ contains a primitive $q$-th root of unity. Let $\xi$ be such a primitive root of unity. Then
$$f(x,\xi^l y)=\sum_{k=0}^{q-1}f_k(x,y^q) \xi^{lk}y^k\ \ \ \  \forall \ k,l.$$
Thus we have
$$\wdt{f}=V(\xi)F$$
where $\wdt{f}$ is the vector with entries $f(x,\xi^ly)$, $1\leq l\leq q$, $F$ is the vector with entries $f_0(x,y^q)$, $yf_1(x,y^q)$, $\cdots$, $y^{q-1}f_{q-1}(x,y^q)$ and $V(\xi)$ is the Vandermonde matrix 
$$\left[\begin{array}{ccccc}
1 & \xi  & \xi^2 & \cdots & \xi^{q-1}\\
1 & \xi^2  & \xi^4 & \cdots & \xi^{2(q-1)}\\
1 & \xi^3  & \xi^6 & \cdots & \xi^{3(q-1)}\\
\vdots & \vdots & \vdots & \ddots & \cdots\\
1 & \xi^q  & \xi^{2q} & \cdots & \xi^{(q-1)q}\\
\end{array}\right].$$
Thus 
$$F=V(\xi)^{-1}\wdt{f}$$
Since the entries of $V(\xi)^{-1}$ are in $\k$ and $\h(f(x,\xi^ly))=\h(f(x,y))$, by Lemma \ref{inequalities_height} \eqref{2} and \eqref{1} we have
$$\h(F)\leq q\Deg(f)^q\h(f),$$
$$\Deg(F)\leq \Deg(f)^q.$$
Thus by Lemma \ref{inequalities_height} \eqref{6}

$$\h(f_i(x,y^q))\leq q\Deg(f)^q\h(f)+\Deg(f)^q(q-1)=\Deg(f)^q(q\h(f)+q-1),$$

$$\Deg(f_i(x,y^q))\leq \Deg(f)^q\ \ \ \forall i.$$\\
(2) If $q=p>0$, then we have
$$\frac{\partial f}{\partial y}=f_1+2f_2y+\cdots+(p-1)f_{p-1}y^{p-2},$$
$$\cdots\cdots\cdots$$
$$\frac{\partial^{p-1}f}{\partial y^{p-1}}=(p-1)!f_{p-1}.$$
Thus we have 
$$\D f=M \wdt{f}$$
where $\D f$ is the vector of entries $\frac{\partial^k f}{\partial y^k}$, for $0\leq k\leq p-1$, $\wdt{f}$ is the vector with entries $f_l(x,y^p)$, for $0\leq l\leq p-1$, and $M$ is a upper triangular matrix with entries in $\k[y]$ and whose determinant is in $\k$. We can check that the $(p-1)\times(p-1)$ minors of $M$ are polynomials of degree $\leq \frac{p(p-1)}{2}$. Thus the height of the coefficients of $M^{-1}$ is less than $\frac{p(p-1)}{2}$. Since
$$\wdt{f}=M^{-1}\D f,$$
by Lemma \ref{inequalities_height} \eqref{2} we obtain

\begin{equation*}\begin{split}\h(f_k(x,y^p))\leq p\Deg(f)\Deg\left(\frac{\partial f}{\partial y}\right)\cdots\Deg\left(\frac{\partial^{p-1} f}{\partial y^{p-1}}\right)\times &\\
\left(\max_{0\leq i\leq p-1}\right.&\left.\left\{\h\left(\frac{\partial^{i}f}{\partial y^{i}}\right)\right\}+\frac{p(p-1)}{2}\right).\end{split}\end{equation*}
Thus by Lemma \ref{height_derivation} we have

$$\h(f_k(x,y^p))\leq p\Deg(f)^p\left(\max_{0\leq i\leq p-1}\left\{\h\left(\frac{\partial^{i}f}{\partial y^{i}}\right)\right\}+\frac{p(p-1)}{2}\right).$$
By applying Lemma \ref{height_derivation} $p-1$ times we obtain

$$\h\left(\frac{\partial^{p-1}f}{\partial y^{p-1}}\right)\leq 4^{p-1}\Deg(f)^{(2\Deg(f)+4)(p-1)}\h(f).$$
Thus we have

$$\h(f_k(x,y^p))\leq p4^{p-1}\Deg(f)^{2(p-1)\Deg(f)+5p-4}\left(\h(f)+\frac{p(p-1)}{2}\right).$$
Moreover, still by Lemma \ref{height_derivation} we have
$$\Deg(f_k(x,y^p))\leq \Deg(f) \ \ \ \forall k.$$\\
(3) If $q=rp^e$ where $\gcd(r,p)=1$ and $e>0$, we write

$$f=\wdt f_{0}(x,y^p)+\wdt f_{1}(x,y^p)y+\cdots+\wdt f_{p-1}(x,y^p)y^{p-1}$$

$$\wdt f_{i}(x,y^p)=\wdt f_{i,0}(x,y^{p^2})+\wdt f_{i,1}(x,y^{p^2})y^p+\cdots+\wdt f_{i,p-1}(x,y^{p^2})y^{p(p-1)}$$

$$\wdt f_{i,j}(x,y^{p^2})= \wdt f_{i,j,0}(x,y^{p^3})+\wdt f_{i,j,1}(x,y^{p^3})y^{p^2}+\cdots+\wdt f_{i,j,p-1}(x,y^{p^3})y^{p^2(p-1)} $$

$$\cdots\cdots\cdots$$
$$\wdt f_{i_1,\cdots,i_{e-1}}(x,y^{p^{e-1}})=\wdt f_{i_1,\cdots,i_{e-1},0}(x,y^{p^{e}})+\cdots+\wdt f_{i_1,\cdots,i_{e-1},p-1}(x,y^{p^{e}})y^{p^{e-1}(p-1)}$$

$$\wdt f_{i_1,\cdots,i_e}(x,y^{p^e})= \wdt f_{i_1,\cdots,i_e,0}(x,y^q)+\cdots+\wdt f_{i_1,\cdots,i_e,r-1}(x,y^q)y^{p^e( r-1)}.$$
\\
Then by (2) we obtain, for $k\leq e$,
$$\Deg(\wdt f_{i_1,\cdots,i_k}(x,y^p))\leq \Deg(\wdt f_{i_1,\cdots,i_{k-1}}(x,y)),$$
\begin{equation*}\begin{split}\h(\wdt f_{i_1,\cdots,i_k}(x,y^p))\leq  p4^{p-1}\Deg(\wdt f_{i_1,\cdots,i_{k-1}}(x,y))^{2(p-1)\Deg(\wdt f_{i_1,\cdots,i_{k-1}}(x,y))+5p-4}\times\hspace{2cm} \\
\left(\h(\wdt f_{i_1,\cdots,i_{k-1}}(x,y))+\frac{p(p-1)}{2}\right).\end{split}\end{equation*}
Thus by Lemma \ref{height_power} we have
\begin{equation*}\begin{split}\frac{1}{p^{k-1}}\h(\wdt f_{i_1,\cdots,i_k}(x,y^{p^k}))\leq  p4^{p-1}\Deg(\wdt f_{i_1,\cdots,i_{k-1}}(x,y^{p^{k-1}}))^{2(p-1)\Deg(\wdt f_{i_1,\cdots,i_{k-1}}(x,y^{p^{k-1}}))+5p-4}\times\\
\left(\h(\wdt f_{i_1,\cdots,i_{k-1}}(x,y^{p^{k-1}}))+\frac{p(p-1)}{2}\right).\end{split}\end{equation*}
\\
By (1) we obtain
$$\Deg(\wdt f_{i_1,\cdots,i_{e+1}}(x,y^r))\leq \Deg(\wdt f_{i_1,\cdots,i_{e}}(x,y))^r,$$
$$\h(\wdt f_{i_1,\cdots,i_{e+1}}(x,y^r))\leq \Deg(\wdt f_{i_1,\cdots,i_e}(x,y))^r(r\h(\wdt f_{i_1,\cdots,i_e}(x,y))+r-1)$$
and, by Lemma \ref{height_power},
$$\frac{1}{p^e}\h(\wdt f_{i_1,\cdots,i_{e+1}}(x,y^q))\leq \Deg(\wdt f_{i_1,\cdots,i_e}(x,y^{p^e}))^r(r\h(\wdt f_{i_1,\cdots,i_e}(x,y^{p^e}))+r-1)$$
\\
Since the power series $f_i(x,y^q)$ of the statement of the lemma are expressed by  the power series 
$$\wdt f_{i_1,\cdots,i_{e+1}}(x,y^q),$$
 by induction and Lemma \ref{height_power} we deduce 
$$\Deg(f_i(x,y^q))\leq \Deg(f)^r,$$
$$\h(f_i(x,y^q))\leq   p^{\frac{e(e+1)}{2}}q4^{p^e-1}\Deg(f)^{r+2(p-1)e\Deg(f)+(5p-4)e}\left(r\h(f)+(e+1)\frac{r(r-1)}{2}\right)$$
$$\leq q^2p^{\frac{e(e+1)}{2}}4^q\Deg(f)^{2q\Deg(f)+5q}\left(\h(f)+\frac{q(q-1)}{2}\right).$$
\end{proof}

%
%

\section{Effective Weierstrass Division Theorem}\label{W}
In this part we prove an effective Weierstrass Division Theorem for algebraic power series. The proof (thus the complexity) is more complicated in the positive characteristic case since the Weierstrass polynomial associated to the divisor  $f$ may have irreducible factors that are not separable. The proof we give here is essentially the same as the one given in \cite{La}.
  
  \begin{lemma}[Weierstrass Preparation Theorem]\label{height_preparation}
  Let $\k$ be any field.
  Let $f$ be an algebraic power series which is $x_n$-regular of order $d$. Then there exist a unit $u\in\k\lag x\rag$ and a Weierstrass polynomal $P\in\k\lag x'\rag[x_n]$ such that 
  $$f=u\cdot P$$
  and 
  $$\Deg(P)\leq \h(f)!,$$
  $$\h(P)\leq 2d\h(f)^{d+1}.$$
  
  \end{lemma}
  \vspace{0.2cm}
  \begin{proof}
  The existence of $u$ and $P$ comes from the Weierstrass Preparation Theorem for formal power series.\\
  Let $\a_1$, $\cdots$, $\a_d\in\K$ be the roots of $P(x_n)$ counted with multiplicities. Then  we have $P=\prod_{i=1}^d(x_n-\a_i)$. By Remark \ref{roots} the roots of $P(x_n)$ are the roots of $f$ thus, by Lemma \ref{height_lemma}, $P$ is an algebraic power series. Hence  $u$ is also an algebraic power series. \\
By Lemma \ref{inequalities_height} \eqref{3} $\h(x_n-\a_i)\leq \h(\a_i)+\Deg(\a_i)$ and $\Deg(x_n-\a_i)=\Deg(\a_i)\leq \h(f)$ for all $i$ by Lemma \ref{height_lemma}. Thus, by Lemma \ref{inequalities_height} \eqref{5},
  $$\h(P)\leq d\cdot \Deg(\a_1)\cdots\Deg(\a_d)\cdot \max_i\{\h(\a_i)+\Deg(\a_i)\}\leq d\h(f)^d(\h(f)+\h(f)).$$
  Moreover $P\in\k(x,\a_1,\cdots,\a_d)$. But $[\k(x,\a_1,\cdots,\a_d):\k(x)]\leq \h(f)!$ by Lemma \ref{height_lemma} hence
    $$\Deg(P)\leq \h(f)!$$
  \end{proof}

\begin{lemma}\label{height_division1}
  Let $f$ be an algebraic power series which is $x_n$-regular of order $d$ and let us assume that $f$ has $d$ distinct roots in $\K$. Let $g$ be any algebraic power series. Then there exist unique algebraic power series $q$ and $r$ such that $r\in\k\lag x'\rag[x_n]$ is of degree $<d$ in $x_n$ and
  $$g=fq+r.$$
  Moreover, if $r=r_0+r_1x_n+\cdots+r_{d-1}x_n^{d-1}$, we have 
  $$\h(r_i)\leq 4d(\h(f)!)^{d+1}\h(f)^2\Deg(g)\max\left\{d!\frac{d(d-1)}{2}\h(f)^{\frac{d(d-1)}{2}} (\h(f)!)^{d!+2}, \h(g)\right\}$$
    $$\leq 4 \h(f)^{\h(f)^{O(d)}} \Deg(g)(\h(g)+1)\ \ \ \forall i,$$
    where $O(d)$ denotes a function of $d$ bounded by a linear function in $d$,
  $$\h(r)\leq d \left( \h(f)!\Deg(g)^d\right)^d(\max_i\{\h(r_i)\}+d-1)$$
  $$\text{ and }\ \Deg(r_i),\ \Deg(r)\leq   \h(f)!\Deg(g)^d\ \ \ \forall i.$$
    \end{lemma}
    \vspace{0.2cm}
 \begin{proof}
 The Weierstrass Division Theorem for algebraic power series is well known (see \cite{La}), the only improvement  is the inequalities on the heights and degrees. \\
The Weierstrass Division Theorem for formal power series gives the existence and unicity of $q$ and $r$. Thus we have to show that $q$ and $r$ are algebraic and to prove the bounds on the heights and degrees. 
 Let $\a_1$, $\cdots$, $\a_d\in \K$ be the roots of $f$. Then we have
$$g(x',\a_i)=r(x',\a_i) \ \ \ \  \forall i.$$
By writing $r=r_0+r_1x_n+\cdots+r_{d-1}x_n^{d-1}$ with $r_j\in\k\lb x'\rb$ for all $j$, we obtain:
$$V(\a)\wdt{r}=\wdt{g}(\a)$$
where $V(\a)$ is the $d\times d$ Vandermonde matrix of the $\a_i$:
$$\left[\begin{array}{ccccc}
1 & \a_1 & \a_1^2 & \cdots & \a_1^{d-1}\\

1 & \a_2 & \a_2^2 & \cdots & \a_2^{d-1}\\
\vdots & \vdots & \vdots & \vdots & \vdots\\

1 & \a_d & \a_d^2 & \cdots & \a_d^{d-1}
\end{array}\right],$$

\noindent $\wdt{r}$ is the $d\times 1$ column vector with entries $r_k$, and $\wdt{g}(\a)$ is the $d\times 1$ column vector with entries $g(x',\a_j)$. Since the $\a_i$ are distinct $V(\a)$ is invertible and
 we obtain 
\begin{equation}\label{vdm} \wdt{r}=V(\a)^{-1}\wdt{g}(\a).\end{equation}
By Lemmas \ref{height_lemma} and  \ref{height_compo1} we see the $g(x',\a_j)$ are algebraic. Then Equality \eqref{vdm} shows that  the $r_i$ and $r$ are algebraic power series, thus $q$ is also an algebraic power series. Again by Lemmas \ref{height_lemma} and \ref{height_compo1} we have for all $i$:
$$\h(g(x',\a_i))\leq 2\h(g)\cdot\h(f),$$
$$\Deg(g(x',\a_i))\leq \h(f)\cdot\Deg(g).$$\\
The determinant of $V(\a)$ is the sum of $d!$ elements of the form 
$$\a_{\s(0)}^0\a_{\s(1)}^1\a_{\s(2)}^2\cdots\a_{\s(d-1)}^{d-1},$$
 where $\s$ is a permutation of $\{0,\cdots, d-1\}$. Each of these elements belongs to $\k(x',\a_1, \cdots,\a_d)$ so their degree is bounded  $\h(f)!$ by Lemma \ref{height_lemma}. Again by Lemma \ref{height_lemma}  $\h(\a_i)\leq \h(f)$ and $\Deg(\a_i)\leq \h(f)$ for any $i$, thus by Lemma \ref{inequalities_height} \eqref{5} we see that for any permutation $\s$ we have:
 $$\h(\a_{\s(0)}^0\a_{\s(1)}^1\a_{\s(2)}^2\cdots\a_{\s(d-1)}^{d-1})\leq \frac{d(d-1)}{2}\h(f)^{\frac{d(d-1)}{2}+1}.$$
 Thus  by Lemma \ref{inequalities_height} \eqref{2} we have 
 $$\h(\det(V(\a)))\leq d!\frac{d(d-1)}{2}\h(f)^{\frac{d(d-1)}{2}+1}(\h(f)!)^{d!}.$$
 \\
The entries of $V(\a)^{-1}$ are $(d-1)\times (d-1)$ minors of $V(\a_i)$ divided by $\det(V(\a))$. Exactly as above the height of such an $(d-1)\times(d-1)$ minor is bounded by
$$(d-1)!\frac{(d-1)(d-2)}{2}\h(f)^{\frac{(d-1)(d-2)}{2}+1}(\h(f)!)^{(d-1)!}$$
and its degree is bounded by $\h(f)!$ since it is an element of $\k(x',\a_1, \cdots,\a_d)$ (see Lemma \ref{height_lemma}).
Hence by Lemma \ref{inequalities_height} \eqref{5}  the height of the entries of $V(\a)^{-1}$ is bounded by
 $$\h_V:=2d!(\h(f)!)^2(\h(f)!)^{d!}\frac{d(d-1)}{2}\h(f)^{\frac{d(d-1)}{2}+1} =$$
 $$=2d!\frac{d(d-1)}{2} \h(f)^{\frac{d(d-1)}{2}+1}(\h(f)!)^{d!+2}.$$
 Moreover their degree is bounded by $\h(f)!$
 since they belong to $\k(x,\a_1, \cdots,\a_d)$.\\
 \\
 If $v$ is an entry of $V(\a)^{-1}$ Lemma \ref{inequalities_height} \eqref{5} shows
 
 $$\h(vg(x',\a_i))\leq 2\h(f)! \Deg(g(x',\a_i))\max\{\h_V,\h(g(x',\a_i))\} \ \ \forall i.$$
Since $r_j$ is of the form $v_1g(x',\a_1)+\cdots+v_dg(x',\a_d)$ where $v_1$, $\cdots$, $v_d$ are entries of $V(\a)^{-1}$ (by Equation \eqref{vdm}) we obtain from Lemma \ref{inequalities_height} \eqref{2}:
$$\h(r_j)\leq d(\h(f)!)^{d}\max_i\left\{2\h(f)!\Deg(g(x',\a_i))\max\{\h_V,\h(g(x',\a_i))\}\right\}. $$
Hence Lemmas \ref{height_lemma} and \ref{height_compo1} show

\begin{equation*}\begin{split}\h(r_j)\leq 4d(\h(f)!)^{d+1}\h(f) &\Deg(g)\times \\ 
\max&\left\{d!\frac{d(d-1)}{2} \h(f)^{\frac{d(d-1)}{2}+1}(\h(f)!)^{d!+2}, \h(f)\h(g)\right\}\end{split}\end{equation*}
$$=4d(\h(f)!)^{d+1}\h(f)^2\Deg(g)\max\left\{d!\frac{d(d-1)}{2}\h(f)^{\frac{d(d-1)}{2}} (\h(f)!)^{d!+2}, \h(g)\right\}.$$
Moreover  $r_j$ and $r\in\k(x',\a_1,\cdots,\a_d,g(x',\a_1),\cdots,g(x',\a_d))$, hence we have (by Corollary \ref{aschen}):
$$\Deg(r_j)\leq   \h(f)!\Deg(g)^d,\ \ \Deg(r)\leq \h(f)!\Deg(g)^d.$$ 
Since 
$$r=r_0+x_nr_n+\cdots+x_n^{d-1}r_{d-1},$$
$$H(r)\leq d \left( \h(f)!\Deg(g)^d\right)^d(\max_i\{\h(r_i)\}+d-1) $$
by Lemma \ref{inequalities_height} \eqref{2}.
 \end{proof}
 
 \begin{lemma}\label{height_division2}
Let assume that $\k$ is a field of characteristic $p>0$. Let $f$ be an irreducible algebraic power series which is $x_n$-regular of order $d$ and let us assume that its Weierstrass polynomial is not separable. Let $g$ be any algebraic power series. Then there exist unique algebraic power series $q$ and $r$ such that $r\in\k\lag x'\rag[x_n]$ is of degree $<d$ in $x_n$ and
  $$g=fq+r.$$
  Moreover, if $r=r_0+r_1x_n+\cdots+r_{d-1}x_n^{d-1}$, we have 
 $$ \h(r_i)\leq (2\h(f))^{(2\h(f))^{O(d)}}\Deg(g)^{2d(\Deg(g)+2)}(\h(g)+1) \ \ \ \forall i,$$
  $$\h(r)\leq (2\h(f))^{(2\h(f))^{O(d)}}\Deg(g)^{O(d\,\Deg(g))}(\h(g)+1)\ \ \ \forall i,$$
  $$\Deg(r_i),\ \Deg(r)\leq \h(f)!\Deg(g)^d \ \ \forall i.$$

 \end{lemma}
 
 \begin{proof}
 Let $P$ denote the Weierstrass polynomial of $f$. Since $f$ is an irreducible power series  then $P$ is an irreducible monic polynomial of $\k\lb x'\rb[x_n]$ hence $P$ is an irreducible polynomial of $\k((x'))[x_n]$. Then we can write
 $$P=\prod_{k=1}^D(x_n-\a_i)^{p^e}$$
 where $\a_1$, $\cdots$, $\a_D$ are the distinct roots of $P(x_n)$ in $\K$ and $e$ is a positive integer. Thus $P\in\k\lag x'\rag[x_n^{p^e}]$ by Lemma \ref{height_preparation} and $d=Dp^e$. By the Weierstrass Division Theorem for formal power series we have
 $$g=Pq+r$$
 where $$r=r_0+r_1x_n+\cdots+r_{d-1}x_n^{d-1}$$
 and $r_i\in\k\lb x'\rb$. Let us write
 $$g=g_0(x',x_n^{p^e})+g_1(x',x_n^{p^e})x_n+\cdots+g_{p^e-1}(x',x_n^{p^e})x_n^{p^e-1}$$
 where $g_i:=g_i(x',x_n^{p^e})\in\k\lag x',x_n^{p^e}\rag$ for all $i$ by  Lemma \ref{height_extraction} . \\
 We define $\wdt{P}$  by
 $$\wdt{P}(x',x_n^{p^e})=P(x',x_n).$$
 Then $\wdt{P}(x',x_n)$ is a Weierstrass polynomial in $x_n$ of degree $D$ with algebraic power series coefficients  and $\h(\wdt{P}(x',x_n))\leq\h(P(x',x_n))$ by Lemma \ref{height_power}.
  Let us perform the Weierstrass Division of $g_i(x',x_n)$ by $\wdt{P}$:
 $$g_i(x',x_n)=\wdt{P}q_i+\sum_{j=0}^{D-1}r_{i,j}(x')x_n^j.$$
 By Lemma \ref{height_division1} the $r_{i,j}(x')$ are algebraic power series and
 \begin{equation}\begin{split}\h(r_{i,j})\leq 4D(\h(P)!)^{D+1}&\h(P)^2\Deg(g_i(x',x_n))\\
 \max&\left\{D!\frac{D(D-1)}{2}\h(P)^{\frac{D(D-1)}{2}} (\h(P)!)^{D!+2}, \h(g_i(x',x_n))\right\}.\end{split}\end{equation}
By Lemma \ref{height_power} we have $\Deg(g_i(x',x_n))\leq \Deg(g(x',x_n^{p^e}))$ for every $i$ thus, by  Lemma \ref{height_extraction}, we have $\Deg(g_i(x',x_n))\leq \Deg(g)$. Again by Lemma \ref{height_power} we have $\h(g_i(x',x_n))\leq \h(g(x',x_n^{p^e}))$. Moreover by Lemma \ref{height_preparation} $\h(P)\leq 2d\h(f)^{d+1}$. Thus we obtain (by using Lemma \ref{height_extraction} and since $D\leq d$, $p^e\leq d$ and $d\leq \h(f)$ by Lemma \ref{height_ord})
 \begin{equation}\begin{split} \h(r_{i,j})
 \leq 4d((2d\h(f)&^{d+1})!)^{d+1}(2d\h(f)^{d+1})^2\Deg(g)\times\\
 \\ \times\max\left\{d!\frac{d(d-1)}{2}\right.&(2d\h(f)^{d+1}))^{\frac{d(d-1)}{2}} ((2d\h(f)^{d+1})!)^{d!+2},\\ 
 &\left.p^{2e}p^{e(e+1)/2}4^{p^e}\Deg(g)^{2p^e\Deg(g)+5p^e}\left(\h(g)+\frac{p^e(p^e-1)}{2}\right)\right\}\end{split}\end{equation}
 $$\leq (2\h(f))^{(2\h(f))^{O(d)}}\Deg(g)^{2d(\Deg(g)+2)}(\h(g)+1).$$
 Finally, since 
 $$g_i(x',x_n^{p^e})=P\,q_i(x',x_n^{p^e})+\sum_{j=0}^{D-1}r_{i,j}(x')x_n^{jp^e},$$
 then
 $$r=\sum_{i=0}^{p^e-1}\sum_{j=0}^{D-1}r_{i,j}(x')x_n^{jp^e+i}$$
 by unicity of the remainder in the Weierstrass division. Thus 
 Lemma  \ref{inequalities_height} \eqref{2} shows
 $$\h(r)\leq p^eD\cdot\Deg(r_{i,j}(x'))^{p^eD}\max_{i,j}\{\h(r_{i,j}(x'))+jp^e\}.$$
  Moreover 
 $$\Deg(r_{i,j}), \Deg(r)\leq \h(f)!\Deg(g_i)^D \ \ \ \forall i,j$$
 since $r_{i,j}$ and $r\in\k(x',\a_1,\cdots,\a_D,g_i(x',\a_1),\cdots,g_i(x',\a_D))$ (as shown in the proof of Lemma \ref{height_division1}). Hence
$$\h(r)\leq  (2\h(f))^{(2\h(f))^{O(d)}}\Deg(g)^{O(d\,\Deg(g))}(\h(g)+1). $$
 \end{proof}
 
 We will use at several places this basic lemma:

 \begin{lemma}\label{lemma_tech}
 For  any $\e>0$, $a>0$ and $d\in\N$ we have
 $$(2d)^{(2d)^{ad}}\leq 2^{2^{O(d^{1+\e})}}.$$
 \end{lemma}
 \begin{proof}
 Let $a>0$ and $\e>0$. 
There exists a constant $C>0$ such that for any $d$ large enough we have:
 $$ad\ln(2d)+\ln(\ln(2d))\leq C\ln(2)d^{1+\e}+\ln(\ln(2)).$$
 Thus
 $$ (2d)^{ad}\ln(2d)\leq \ln(2)2^{Cd^{1+\e}}$$
 and 
 $$(2d)^{(2d)^{ad}} \leq 2^{2^{Cd^{1+\e}}}.$$
 \end{proof}

 
 \begin{theorem}[Weierstrass Division Theorem]\label{division_sep}
 Let $\k$ be a field. Let $f$ be an algebraic power series which is $x_n$-regular of order $d$. Let $g$ be an algebraic power series. Then there exist unique algebraic power series $q$ and $r$ such that $r\in\k\lag x'\rag[x_n]$ is of degree $<d$ in $x_n$:
 $$r=r_0+r_1x_n+\cdots r_{d-1}x_n^{d-1},\ \ r_i\in\k\lag x'\rag \ \ \forall i$$ 
 and
  $$g=fq+r.$$
  Moreover we have the following bounds (for any $\e>0$):
  \begin{itemize}
  \item[i)] if  char$(\k)=0$:
 $$\h(r)\leq2^{2^{O(\h(f)^{1+\e})}}\Deg(g)^{d^4+d^3+6d^2-5d+3}(\h(g)+1),$$
 $$\h(r_i)\leq 2^{2^{O(\h(f)^{1+\e})}}\Deg(g)^{O(d^4) }(\h(g)+1)\ \ \ \forall i,$$
 $$\h(q)\leq  2^{2^{O(\h(f)^{1+\e})}}\Deg(g)^{d^4+d^3+6d^2-3d+5 }\Deg(f)(\h(g)+1).$$

\item[ii)] if $\cha(\k)>0$:
$$\h(r)\leq2^{2^{O(\h(f)^{1+\e})}}\Deg(g)^{O(d^4\Deg(g)^4)}(\h(g)+1),$$
$$\h(r_i)\leq 2^{2^{O(\h(f)^{1+\e})}}\Deg(g)^{O(d^4\Deg(g)^4) }(\h(g)+1)\ \ \ \forall i,$$
 $$\h(q)\leq  2^{2^{O(\h(f)^{1+\e})}}\Deg(g)^{O(d^4\Deg(g)^4)}\Deg(f)(\h(g)+1).$$

\end{itemize}
In both cases we have 
$$\Deg(r)\leq\h(f)!\Deg(g)^d,$$
$$\Deg(r_i)\leq \h(f)!\Deg(g)^d \ \ \forall i,$$ 
$$\Deg(q)\leq \h(f)!\Deg(g)^{d+1}\Deg(f).$$\\
 \end{theorem}

\begin{proof}
Let us write $f=u.P$ where $u$ is a unit and $P$ a Weierstrass polynomial in $x_n$. Let us decompose $P$ into the product of irreducible Weierstrass polynomials
$$P=P_1\cdots P_s.$$
Let us consider the following Weierstrass divisions:
$$g=P_1Q_1+R_1$$
$$Q_1=P_2Q_2+R_2$$
$$\cdots\cdots$$
$$Q_{s-1}=P_sQ_s+R_s.$$
Then
$$g=P_1\cdots P_sQ_s+R_1+P_1R_2+P_1P_2R_3+\cdots+P_1\cdots P_{s-1}R_s.$$
Thus, by unicity of the Weierstrass division, we have
$$u\cdot q=Q_s,$$
$$r:=R_1+P_1R_2+P_1P_2R_3+\cdots+P_1\cdots P_{s-1}R_s$$
are the quotient and the the remainder of the division of $g$ by $P$.\\
Here $s\leq d$ since $P$ is monic of degree $d$ in $x_n$. Let $d_i$ be the degree in $x_n$ of the polynomial $P_i$ for $1\leq i\leq s$. Let us choose $1\leq i\leq s$ and let us denote by $\a_1$, $\cdots$, $\a_{d_i}\in\K$ the roots of $P_i$.\\
First let us prove the lemma when   char$(\k)=0$. In this case these roots are distinct. Then

$$P_i=\prod_{i=1}^{d_i}(x_n-\a_i).$$
We have $\h(x_n-\a_i)\leq \h(\a_i)+\Deg(\a_i)\leq 2\h(f)$ (by Lemma \ref{inequalities_height} \eqref{3}) and $\Deg(x_n-\a_i)=\Deg(\a_i)\leq \h(f)$. Then, by Lemma \ref{inequalities_height} \eqref{5} and since $d_i\leq d\leq \h(f)$ (by Lemma \ref{height_ord}), we have
$$\h(P_i)\leq d_i\h(f)^{d_i}\cdot 2\h(f)\leq 2\h(f)^{\h(f)+2}.$$
Moreover
$$\Deg(P_i)\leq \h(f)!$$
since $P_i$ is in the extension of $\k(x)$ generated  by the roots of $f$.\\
\\
 \noindent Exactly as in the proof of Lemma \ref{height_division1} we have
 $$R_i\in\k(x,\a_1,\cdots,\a_d, Q_{i-1}(x',\a_1),\cdots, Q_{i-1}(x',\a_d)).$$
 Since $Q_{i-1}=\frac{Q_{i-2}-R_{i-1}}{P_{i-1}}$ we obtain, by induction,
 $$Q_{i-1}(x',\a_k)\in\k(x',\a_1,\cdots,\a_d,Q_{i-2}(x',\a_1),\cdots, Q_{i-2}(x',\a_d))$$
 thus
 \begin{equation}\label{extension}R_i, Q_i, P_i\in\k(x,\a_1,\cdots,\a_d,g(x',\a_1),\cdots,g(x',\a_d)) \ \ \forall i\end{equation}
 and 
 $$\Deg(R_i), \Deg(Q_i), \Deg(r), \leq \h(f)!\Deg(g)^d\ \ \ \forall i$$
 by Corollary \ref{aschen}. Since $q=\frac{g-r}{f}$, then 
 $$q\in \k(x,\a_1,\cdots,\a_d,g(x',\a_1),\cdots,g(x',\a_d),g,f)$$ 
 and $\deg(q)\leq \h(f)!\Deg(g)^{d+1}\Deg(f)$.
 Thus the inequalities on the degrees are proven.\\
\\
\\
Let $\e$ be a positive real number.
By Lemma \ref{height_division1}  the height of $R_1$ is bounded by
$$d_1(\h(P_1)!\Deg(g)^{d_1})^{d_1}(4\h(P_1)^{\h(P_1)^{O(d_1)}}\Deg(g)(\h(g)+1)+d_1-1)$$
and so we obtain

\begin{equation}\label{R1}\h(R_1)\leq 2^{2^{O(\h(f)^{1+\e})}}\cdot \Deg(g)^{d^2+1}(\h(g)+1) \end{equation}
by Lemma \ref{lemma_tech} since $\h(P_1)\leq 2\h(f)^{\h(f)+2}$ and $d_1\leq d\leq \h(f)$. \\
\\
By  Lemma \ref{inequalities_height} \eqref{2} and \eqref{5} we have
 \begin{equation*}\begin{split}\h(Q_1)=\h\left(\frac{g-R_1}{P_1}\right)\leq 2\Deg(P_1)&\Deg(g-R_1)\times\\
 \times \max\{\h(P_1),2&\Deg(g)\Deg(R_1)\max\{\h(g),\h(R_1)\}\}\end{split}\end{equation*}
 
 $$\leq 4\h(f)!\Deg(g)^2\Deg(R_1)^2\max\{\h(P_1),\h(g),\h(R_1)\}$$\\
 since $\Deg(P_1)\leq \h(f)!$ and $d\leq \h(f)$. Hence by Lemma \ref{lemma_tech} and the bound \eqref{R1} on $\h(R_1)$ we obtain
 
 \begin{equation}\label{Q1}\h(Q_1)\leq 2^{2^{O(\h(f)^{1+\e})}}\Deg(g)^{d^2+2d+3}(\h(g)+1).\end{equation}\\
 \\
    Still by Lemma \ref{height_division1}, and as we have shown for $\h(R_1)$, we have 
    
 \begin{equation}\label{Ri}\h(R_i)\leq 2^{2^{O(\h(f)^{1+\e})}}\Deg(Q_{i-1})^{d^2+1}(\h(Q_{i-1})+1),\end{equation}
 and by Lemma \ref{inequalities_height} \eqref{2} and \eqref{5}, and as we have done for $\h(Q_1)$, we have
 
 \begin{equation*}\begin{split}\h(Q_i)\leq 2\Deg(P_i)\Deg(Q_{i-1}-&R_i)\times\\
\times \max\{\h(P_i),2&\Deg(R_i)\Deg(Q_{i-1})\max\{\h(R_i),\h(Q_{i-1})\}\}\end{split}\end{equation*}
 $$\leq 4\h(f)!\Deg(Q_{i-1})^2\Deg(R_i)^2\max\{\h(P_i),\h(Q_{i-1}),\h(R_i)\}$$
 $$\leq 4(\h(f)!)^5\Deg(g)^{4d}\max\{\h(P_i),\h(Q_{i-1}),\h(R_i)\}.$$
 The previous bound \eqref{Ri} on $\h(R_i)$ gives 
 
 $$\h(Q_i)\leq 2^{2^{O(\h(f)^{1+\e})}}\Deg(g)^{4d}\Deg(Q_{i-1})^{d^2+1}(\h(Q_{i-1})+1).$$
 Since $d\leq \h(f)$, $\Deg(Q_i)\leq \h(f)!\Deg(g)^d$ for $i$, and by using the bound \eqref{Q1} on $\h(Q_1)$, we obtain by induction on $i$
 
   $$\h(Q_i)\leq 2^{2^{O(\h(f)^{1+\e})}}\Deg(g)^{(d^3+d^2+4d)(i-1)+d^2+2d+3}(\h(g)+1) \ \ \ \ \forall i\geq 1.$$
 Thus the bound \eqref{Ri} gives 
 
 \begin{equation}\label{Ri2}\h(R_i)\leq 2^{2^{O(\h(f)^{1+\e})}}\Deg(g)^{(d^3+d^2+4d)i+d^2-6d+3  }(\h(g)+1) \ \ \ \forall i\geq 2. \end{equation}
 \\
By Lemma \ref{inequalities_height}  \eqref{5}, for all $i\geq 2$

   \begin{equation*}\begin{split}\h(P_1\cdots P_{i-1}R_i)\leq  i \Deg(P_1)\cdots\Deg(P_{i-1})\Deg(R_i)&\times\\
   \max&\{\h(P_1),\cdots,\h(P_{i-1}),\h(R_i)\}\end{split}\end{equation*}

$$\leq i(\h(f)!)^i\Deg(g)^d\max\{\h(P_1),\cdots,\h(P_{i-1}),\h(R_i)\}$$

$$\leq 2^{2^{O(\h(f)^{1+\e})}}\Deg(g)^{(d^3+d^2+4d)i+d^2-5d+3  }(\h(g)+1)$$
by \eqref{Ri2}.
We have $P_i$ and $R_i\in k(x,\a_1,\cdots,\a_d,g(x(,\a_1),\cdots,g(x',\a_d))$ for all $i$, then

$$\Deg(P_1\cdots P_{i-1}R_i)\leq \h(f)!\Deg(g)^d\ \ \ \forall i.$$
\\
  Thus by Lemma \ref{inequalities_height} \eqref{2} we obtain
  
  $$\h(r)\leq s (\h(f)!\Deg(g)^d)^s \cdot 2^{2^{O(\h(f)^{1+\e})}}\Deg(g)^{(d^3+d^2+4d)s+d^2-5d+3  }(\h(g)+1)$$
  
  $$\leq 2^{2^{O(\h(f)^{1+\e})}}\Deg(g)^{d^4+d^3+6d^2-5d+3  }(\h(g)+1)$$
  since $s\leq d$ and $d\leq \h(f)$.
  \\
   Thus by Lemma \ref{inequalities_height} \eqref{2} and \eqref{5} 
  \begin{equation*}\begin{split}\h(q)=\h\left(\frac{g-r}{f}\right)\leq 2\Deg(g-r)&\Deg(f)\\
  \max&\left\{\h(f),2\Deg(g)\Deg(r)\max\{\h(g),\h(r)\}\right\}\end{split}\end{equation*}
$$\leq 4\Deg(g)^2\Deg(r)^2\Deg(f)\cdot 2^{2^{O(\h(f)^{1+\e})}}\Deg(g)^{d^4+d^3+6d^2-5d+3  }(\h(g)+1)$$
$$\leq 2^{2^{O(\h(f)^{1+\e})}}\Deg(g)^{d^4+d^3+6d^2-3d+5 }\Deg(f)(\h(g)+1).$$
\\
If we write 
$r(x)=r_0(x')+r_1(x')x_n+\cdots+r_{d-1}(x')x_n^{d-1}$ we have
$$r_0(x')=r(x',0)$$
and
\begin{equation}\label{indri}r_{i+1}(x')=\left(\frac{r-(r_0+r_1x_n+\cdots+r_ix_n^i)}{x_n^i}\right)(x',0)\ \ \ \forall i\geq 0.\end{equation}
In particular, from \eqref{extension}, we have
$$r_i\in \k(x',\a_1,\cdots,\a_d,g(x',\a_1),\cdots,g(x',\a_d))\ \ \forall i$$
hence $\Deg(r_i)\leq \h(f)!\Deg(g)^d$ for all $i$ by Corollary \ref{aschen}.\\
From \eqref{indri}, Lemma \ref{height_ord} and Lemma \ref{inequalities_height} \eqref{2} we obtain
$$\h(r_{i+1})\leq \h\left(\frac{r-(r_0+r_1x_n+\cdots+r_ix_n^i)}{x_n^i}\right)$$
$$=\h\left(\frac{r}{x_n^i}-\frac{r_0}{x_n^i}-\cdots-\frac{r_{i-1}}{x_n}-r_i\right)$$
$$\leq (i+2)\Deg(r)\Deg(r_0)\cdots\Deg(r_i)\max\{\h(r)+i,\h(r_0)+i,\cdots,\h(r_{i-1})+1,\h(r_i)\}$$
$$\leq (d+1)(\h(f)!\Deg(g)^d)^{d+1}\left(\max\{\h(r),\h(r_0),\cdots,\h(r_{i-1}),\h(r_i)\}+d\right).$$
Thus, by induction on $i$ and using the bound on $\h(r)$ proven above, we see that
$$\h(r_i)\leq 2^{2^{O(\h(f)^{1+\e})}}\Deg(g)^{O(d^4) }(\h(g)+1)\ \ \ \forall i.$$\\
\\
 In the case $\cha(\k)=p>0$ the proof is completely similar  using Lemma \ref{height_division2} instead of Lemma \ref{height_division1} so we skip the details.
\end{proof}

\vspace{0.2cm}

\begin{rmk}
  We could prove directly the Weierstrass Division Theorem from the Weierstrass Preparation Theorem as done in \cite{CL}. But this would give a bound on the height of the remainder which is not linear in $\h(g)$. This linear bound in $\h(g)$ is exactly what we need to prove Theorem \ref{main}.
  \end{rmk}
  
%
%

\section{Ideal membership problem in localizations of polynomial rings}\label{S_IMP}
Before bounding the complexity of the Ideal Membership Problem in the ring of algebraic power series we review this problem in the ring of polynomials and give extensions to localizations of the ring of polynomials that may be of independent interest.\\
\\
Let $\k$ be a field and $x:=(x_1,\cdots,x_n)$. The following theorem is well known (such a result has first been proven by  G. Hermann \cite{H} but a modern and correct proof is given in the appendix of \cite{MM}):

\begin{theorem}\label{IMP}\cite{H}\cite{MM}
Let $\k$ be a infinite field.
Let $M$ be a submodule of $\k[x]^q$ generated by vectors  $f_1$, $\cdots$, $f_p$ whose components are polynomials of degrees less than $d$. Let $f\in\k[x]^q$. Then $f\in M$ if and only if there exist $a_1$, $\cdots$, $a_p\in\k[x]$ of degrees $\leq \deg(f)+(pd)^{2^n}$ such that 
$$f=a_1f_1+\cdots+a_pf_p.$$
\end{theorem}
\vspace{0.3cm}
\noindent
If we work over the local ring $\k[x]_{(x)}$ the situation is a bit different. Saying that $f\in\k[x]^q$ is in $\k[x]_{(x)}M$ is equivalent to say that there exist polynomials $a_1$, $\cdots$, $a_p$ and $u$, $u\notin (x)$, such that 
\begin{equation}\label{uf} uf=a_1f_1+\cdots+a_pf_p.\end{equation}
There exists an analogue of Buchberger algorithm to compute Gr\"obner basis in local rings  introduced by T. Mora  \cite{Mo} but it does not give effective bounds on the degrees of the $a_i$. We can also do the following:\\
Saying that \eqref{uf} is satisfied is equivalent to say that there exist polynomials $a_1$, $\cdots$, $a_p$, $b_1$,.., $b_n$ such that
$$f=a_1f_1+\cdots+a_pf_p+b_1x_1f+\cdots+b_nx_nf.$$
In this case $u=1-\sum_ix_ib_i$.\\
Thus by applying Theorem \ref{IMP}, we see that $f\in \k[x]_{(x)}M$ if and only if \eqref{uf} is satisfied for polynomials $u$, $a_1$, $\cdots$, $a_p$ of degrees $\leq \deg(f)+(\left(p+n)\max\{d,\deg(f)+1\}\right)^{2^n}$. But this bound is not linear in $\deg(f)$ any more, which may be interesting if $f_1$, $\cdots$, $f_p$ are fixed and $f$ varies.\\
Nevertheless we can prove the following result:
\vspace{0.2cm}
\begin{theorem}\label{IMPlocal}
For any $n$, $q$ and $d\in\N$ there exists an integer $\g(n,q,d)$ such that $\g(n,q,d)=(2d)^{2^{O(n+q)}}$ and satisfying the following property:\\
Let $\k$ be an infinite field, $M$ be a submodule of $\k[x_1,\cdots,x_n]^q$ generated by vectors $f_1$, $\cdots$, $f_p$ of degree $\leq d$ and let $f\in\k[x]^q$. Let $P$ be a prime ideal of $\k[x]$. Then $f\in \k[x]_{P}M$ if and only if there exist polynomials $a_1$, $\cdots$, $a_p$ of degrees $\leq \deg(f)+\g(n,q,d)$ and $u$, $u\notin P$, of degree $\leq \g(n,q,d)$ such that
$$uf=a_1f_1+\cdots+a_pf_p.$$
\end{theorem}
\vspace{0.2cm}
\begin{proof}
Let  $R$ be the ring defined as follows (this is the idealization of $M$ - see \cite{Na}): the set $R$ is equal to $\k[x]\times \k[x]^q$ and we define the sum and the product as follows:
$$(p,f)+(p',f'):=(p+p',f+f')$$
$$(p,f).(p',f'):=(pp', pf'+p'f)\ \ \forall (p,f), (p',f')\in \k[x]\times\k[x]^q.$$
Let $I:=\{0\}\times M\subset R$. Then $I$ is an ideal of $R$ and it is generated by $(0,f_1)$, $\cdots$, $(0,f_q)$.\\
 Moreover $R$ is isomorphic to the  ring
$$R':=\frac{ \k[x_1,\cdots,x_n,y_1,\cdots,y_q]}{(y_1,\cdots,y_q)^2}$$
and the isomorphism $\s:R\lgw R'$ is defined as follows:\\
If $(p,f)\in R$, $f:=(f^{(1)},\cdots,f^{(q)})$, then $\s(p,f)$ is the image of $p+f^{(1)}y_1+\cdots+f^{(q)}y_q$ in $R'$.
\\
The image $I$ by $\s$ is  an ideal of $R'$ and we denote by $I'$ an ideal of $\k[x,y]$ whose image in $R'$ is equal to $\s(I)$.
Thus, by identifying $R$ and $R'$, we have the following equivalences:
$$f\in M \Longleftrightarrow (0,f)\in I \Longleftrightarrow f^{(1)}(x)y_1+\cdots+f^{(q)}(x)y_q\in I'+(y)^2.$$ \\
Let us assume that the theorem is proven when $q=1$. We will apply it when $M=I'+(y)^2$ is an ideal of $\k[x,y]$. If we write $f_i=(f_{i,1},\cdots, f_{i,q})$ for $1\leq i\leq p$  then $I'+(y)^2$ is generated by $\wdt{f}_1(x,y):=\sum_{j=1}^qf_{1,j}y_j$, $\cdots$, $\wdt{f}_p(x,y):=\sum_{j=1}^qf_{p,j}y_j$ and the $y_iy_j$ for $1\leq i\leq j\leq q$, whose degrees are less than $d+2$. Thus, by assumption,  there exist $u(x,y)$, $a_1(x,y)$, $\cdots$, $a_p(x,y)$, $a_{i,j}(x,y)$ for $1\leq i\leq j\leq $ with $u(0,0)\neq 0$ and such that 
\begin{equation}\label{eqIMP}u\left(f^{(1)}(x)y_1+\cdots+f^{(q)}(x)y_q\right)=\sum_{i=1}^pa_i\wdt{f}_i+\sum_{1\leq i\leq j\leq q}a_{i,j}y_iy_j\end{equation}
and 
$$\deg(a_k),\ \deg(a_{i,j})\leq \deg(f)+\g(n+q,1,d+2)$$
where $\g(n+q,1,d+2)\leq (2d)^{2^{O(n+q)}}$. By identifying the coefficients of $y_1$,..., $y_q$ of both sides of the Equality  \eqref{eqIMP} we obtain
$$u(x,0)f(x)=\sum_{i=1}^pa_i(x,0)f_i(x)$$ and this proves the theorem. Thus we only need to prove the theorem when $M=I$ is an ideal of $\k[x]$ (i.e. for $q=1$).\\
\\
Let $I=Q_1\cap\cdots \cap Q_s$ be an irredundant primary decomposition of $I$ in $\k[x]$. Let us assume that $Q_1$, $\cdots$, $Q_r\subset P$ and $Q_i\not\subset P$ for $i>r$. Then 
$$I\k[x]_{P}=Q_1\k[x]_{P}\cap\cdots\cap Q_r\k[x]_{P}$$
is an irredundant primary decomposition of $I\k[x]_{P}$ in $\k[x]_{P}$ (see Theorem 17, Chap. 4 \cite{SZ}). Let $J$ be the ideal of $\k[x]$ defined by $J=Q_1\cap\cdots\cap Q_r$. Obviously $I\k[x]_{P}=J\k[x]_{P}$ and moreover for any $f\in\k[x]$, $f\in J\k[x]_{P}$ if and only if $f\in J$. \\
\\
If $r=s$, then $I=J$ and for every $f\in\k[x]$, $f\in I\k[x]_P$ if and only if $f\in I$. So this case is exactly Theorem \ref{IMP}.\\
\noindent
In the general case $r<s$ the problem can also be reduced to Theorem \ref{IMP} as follows.  Each ideal $Q_i$ may be generated by polynomials of degree $\leq (2d)^{2^{O(n)}}$ and this bound depends only on $n$ and $d$ (see Statements 63, 64 and 64 \cite{S}). By Statement 56 of \cite{S}, the ideal $J$ is generated by polynomials of degrees $\leq (2d)^{2^{O(n)}}$ and once more this bound depends only on $n$ and $d$. Let $g_1$, $\cdots$, $g_t$ be such generators of $J$. Since $\deg(g_i)\leq (2d)^{2^{O(n)}}$ for any $i$, then $t$ will be bounded by the number of monomials in $x_1$, $\cdots$, $x_n$ of degree $\leq (2d)^{2^{O(n)}}$, thus $t\leq \binom{(2d)^{2^{O(n)}}+n}{n}\leq (2d)^{2^{O(n)}}$ also.\\
If $f\in I\k[x]_{P}$, then $f\in J$ and by Theorem \ref{IMP}, there exist polynomials $c_1$, $\cdots$, $c_t$ such that 
$$f=c_1g_1+\cdots+c_tg_t$$
where $\deg(c_i)\leq \deg(f)+(td)^{2^n}\leq \deg(f)+ (2d)^{2^{O(n)}}$ for every $i$.\\
Let $J'$ be the ideal of $\k[x]$ equal to $Q_{r+1}\cap \cdots\cap Q_s$. Then as for $J$, $J'$ is generated by polynomials of degrees $\leq (2d)^{2^{O(n)}}$. Since $J'\not\subset P$, one of these generators is not in $P$. Let $u$ be such a polynomial. Then we have $ug_i\in J\cap J'=I$ for every $i$. Thus there exist polynomials  $b_{i,j}$, for $1\leq i\leq t$ and $1\leq j\leq p$, such that
$$ug_i=\sum_jb_{i,j}f_j.$$
Still by Theorem \ref{IMP}, we may choose the $b_{i,j}$ such that $\deg(b_{i,j})\leq (2d)^{2^{O(n)}}$.
Hence
$$uf=\sum_j\left(\sum_ic_ib_{i,j}\right)f_j.$$
Then the result follows since $\deg(u)\leq (2d)^{2^{O(n)}}$ and  
$$\deg\left(\sum_ic_ib_{i,j}\right)\leq \deg(f)+(2d)^{2^{O(n)}}.$$
\end{proof}

\noindent Let $S$ be a multiplicative closed subset of $\k[x]$. The proof of Theorem \ref{IMPlocal} gives also the following result:
\begin{prop}\label{IMP_local2}
Let $\k$ be an infinite field.
Let $M$ be a submodule  of $\k[x_1,\cdots,x_n]^q$ generated by the vectors $f_1$, $\cdots$, $f_p$ and $S$ be a multiplicative closed subset of $\k[x]$. Then there exists a constant $C>0$ (depending only on $M$) such that the following holds:\\
For any $f\in\k[x]^q$,  $f\in S^{-1}M$ if and only if there exist polynomials $a_1$, $\cdots$, $a_p$ of degrees $\leq \deg(f)+C$ and $u$, $u\in S$, of degree $\leq C$ such that
$$uf=a_1f_1+\cdots+a_pf_p.$$
\end{prop}

\begin{proof}
We can adapt the proof of  Theorem \ref{IMPlocal} as follows (we keep the same notations): the reduction to the case where $M=I$ is an ideal of $\k[x]$ remains the same. Then if $I=Q_1\cap\cdots \cap Q_s$ is an irredundant primary decomposition of $I$ in $\k[x]$, we may assume that $Q_1$, $\cdots$, $Q_r\subset \k[x]\backslash S$ and $Q_i\cap S\neq\emptyset$ for $i>r$. Then as before
$$I\cdot S^{-1}\k[x]=Q_1\cdot S^{-1}\k[x]\cap\cdots\cap Q_r\cdot S^{-1}\k[x]$$
is an irredudant primary decomposition of $I\cdot S^{-1}\k[x]$. If $J$ denotes the ideal $Q_1\cap \cdots\cap Q_r$ of $\k[x]$, then for any $f\in \k[x]$, we also have $f\in I\cdot S^{-1}\k[x]=J\cdot S^{-1}\k[x]$ if and only if $f\in J$. \\
Then we follow the proof of Theorem \ref{IMPlocal}: if $f\in\k[x]$ and $f\in I\cdot S^{-1}\k[x]$ then $f\in J$ and there exist polynomials $c_1$, $\cdots$, $c_t$ such that 
$$f=c_1g_1+\cdots+c_tg_t$$
where $\deg(c_i)\leq \deg(f)+(td)^{2^n}\leq \deg(f)+ (2d)^{2^{O(n)}}$ for any $i$ and $g_1$,..., $g_t$ are generators of $J$. Moreover the degrees of the $g_i$ and the integer $t$ are bounded by $(2d)^{2^{O(n)}}$.\\
Now the only difference with the proof of Theorem \ref{IMPlocal} is that $\k[x]\backslash S$ is not an ideal of $\k[x]$. So let us choose a non-zero polynomial $u\in Q_{r+1}\cap\cdots\cap Q_s\cap S$ (such a polynomial exists since $S$ is a multiplicative system and $Q_i\cap S\neq\emptyset$ for all $i>r$) and let us denote by $D$ its degree: $D=\deg(u)$. Then
$ug_i\in I$ for every $i$. \\
Still by following the proof of Theorem \ref{IMPlocal} we see by Theorem \ref{IMP} that there exist polynomials  $b_{i,j}$, for $1\leq i\leq t$ and $1\leq j\leq p$, such that
$$ug_i=\sum_jb_{i,j}f_j$$
with $\deg(b_{i,j})\leq D+(2d)^{2^{O(n)}}$ for every $i$ and $j$. Then
$$uf=\sum_j\left(\sum_ic_ib_{i,j}\right)f_j$$
and
$$\deg\left(\sum_ic_ib_{i,j}\right)\leq \deg(f)+D+(2d)^{2^{O(n)}}.$$
So the proposition is proven with $C=D+(2d)^{2^{O(n)}}.$

\end{proof}



\section{Ideal membership in rings of algebraic power series}\label{S_IMP_alg}

\begin{theorem}\label{IMP_alg}
Let $\k$ be any infinite field. Then there exists two computable functions $C_1(n,q,p,H_1,D_1,D_2)$ and $C_2(n,q,p,H_1,D_1,D_2)$  such that the following holds:\\
Let  $n$, $q$, $p$, $H_1$, $H_2$, $D_1$ and $D_2$ be integers and  $f=(f_1,\cdots,f_q)$ and $g_1=(g_{1,1},\cdots,g_{1,q})$, $\cdots$, $g_p=(g_{p,1},\cdots,g_{p,q})$ be vectors of $\k\lag x_1,\cdots,x_n\rag^q$ satisfying 
$$\h(g_i)\leq H_1\text{ for all i, } \h(f)\leq H_2,$$
$$\left[\k(x,g_{i,j})_{1\leq i\leq p,\ 1\leq j\leq q}:\k(x)\right]\leq D_1,$$
$$\left[\k(x,f_{j})_{ 1\leq j\leq q}:\k(x)\right]\leq D_2.$$
Let us assume that $f$ is in the $\k\lag x\rag$-module generated by the vectors $g_i$. Then there exist
  algebraic power series $a_i$ for $1\leq i\leq p$ such that
\begin{equation}\label{system}f_j=\sum_{i=1}^pa_ig_{i,j},\ \ \ \ 1\leq j\leq q   \end{equation}
and 
$$\h(a_i)\leq C_1\left(n,q,p,H_1,D_1,D_2\right)\cdot (H_2+1) \ \ \forall i,$$
$$\Deg(a_i)\leq C_2(n,q,p,H_1,D_1,D_2)\ \ \forall i.$$
\end{theorem}
\vspace{0.2cm}
\begin{proof}
The theorem is proven by induction on $n$. For $n=0$ and any $q$, $p$, $H_1$, $H_2$, $D_1$, $D_2$ any solution $(a_i)$ of \eqref{system} will have height equal to 0 and degree equal to 1. Let us assume that theorem is proven for an integer $n-1\geq 0$ and any integers $q$, $p$, $H_1$, $H_2$, $D_1$, $D_2$ and let us prove it for $n$.\\
We set $H_g:=\max_{i,j}\h(g_{i,j})$, $D_g:=\max_{i,j}\Deg(g_{i,j})$, $H_f:=\max_j\h(f_j)$ and $D_f:=\max_j\Deg(f_j)$. Let $G$ be the $p\times q$ matrix whose entries are the $g_{i,j}$.
We assume that the rank of $G$ is $q\leq p$ (otherwise some equations may be removed) and that the first $q$ columns are linearly independent. Let 
$\D$ be the determinant of these first $q$ columns.
By a linear change of coordinates me may assume that $\D$ is $x_n$-regular of degree $d$ since $\k$ is infinite. By Lemma \ref{height_ord} $d\leq  \h(\D)$. Moreover $\D$ is a sum of $q!$ elements which are the product of $q$ entries of $G$. Thus by Lemma \ref{inequalities_height} \eqref{2} and \eqref{5} we have

$$\h(\D)\leq q!D_g^{q!}\left(qD_g^q H_g\right)=q!qD_g^{q!+q}H_g.$$ 
Of course $\D\in \k(x,g_{i,j})_{1\leq i\leq p,\ 1\leq j\leq q}$ thus
$$\Deg(\D)\leq D_g.$$
 By Lemma \ref{height_preparation} we can write $\D=u\cdot P$ where $u$ is a unit and $P$ a Weierstrass polynomial of degree $d$ with
$$\h(P)\leq 2d\h(\D)^{d+1}\leq 2\h(\D)^{\h(\D)+2}\leq 2\left(q!q D_g^{q!+q}H_g\right)^{q!q D_g^{q!+q}H_g+2}. $$\\
Set 
$$F_j(x,A):=\sum_{i=1}^pg_{i,j}(x)A_i-f_j(x)\ \ \forall j$$
where $A_1$, $\cdots$, $A_p$ are new variables. \\
Let $a_{i,k}(x')$ be algebraic power series of $\k\lag x'\rag$ for $1\leq i\leq p$ and $0\leq k\leq d-1$. Then  let us set
\begin{equation}\label{a^*}a_i^*:=\sum_{k=0}^{d-1}a_{i,k}(x')x_n^k \ \text{ for } 1\leq i\leq p,\end{equation}

$$a^*:=(a_1^*,\cdots,a_p^*).$$
Let $A_{i,k}$, $1\leq i\leq p$, $0\leq k\leq d-1$, be new variables and let us set
$$A_i^*:=\sum_{k=0}^{d-1}A_{i,k}x_n^k,\ \ \ \ 1\leq i\leq p$$
and
$$A^*:=(A_1^*,\cdots,A_p^*).$$\\
Let us consider the Weierstrass division of $F_j(x,A^*)$ by $\D$ with respect to the variable $x_n$:
$$F_j(x,A^*)=\D.Q_j(x,A^*)+R_j$$
where $$R_j=\sum_{l=0}^{d-1}R_{j,l}(x',A^*)x_n^l.$$
Let us consider the following Weierstrass divisions:
$$g_{i,j}(x)x_n^k=\D. Q_{i,j,k}(x)+R_{g_{i,j}},$$
$$\text{where }R_{g_{i,j}} =\sum_{l=0}^{d-1}R_{i,j,k,l}(x')x_n^l,$$
$$\text{ and } f_j(x)=\D. Q'_j(x)+R_{f_j},$$
$$\text{where } R_{f_j}=\sum_{l=0}^{d-1}R_{j,l}'(x')x_n^l.$$
By unicity of the remainder and the quotient of the Weierstrass division we obtain:
\begin{equation}\label{Qj}Q_j(x,A^*)=\sum_{i=1}^p\sum_{k=0}^{d-1}Q_{i,j,k}(x)A_{i,k}-Q'_j(x),\end{equation}
$$R_{j,l}(x',A^*)=\sum_{i=1}^p\sum_{k=0}^{d-1}R_{i,j,k,l}(x')A_{i,k}-R'_{j,l}(x'),$$
Hence $Q_j(x',A^*)$ and $R_{j,l}(x',A^*)$ are linear with respect to the variables $A_{i,k}$.\\
\\
If 
\begin{equation}\label{eqn-1}R_{j,l}(x',a^*)=0\ \text{for all } j \text{ and }l,\end{equation}
 then $F_j(x,a^*)\in (\D)$  $\forall j.$
This means that there exists a vector of $\k\lag x\rag^q$, denoted by $b(x)$, such that

\begin{equation}\label{eqlin} G(x).a^*(x)-f(x)=\D(x).b(x)\end{equation}
where $G(x)$ is the $q\times p$ matrix with entries $g_{i,j}(x)$ and $f(x)$ is the vector with entries $f_j(x)$. In fact we can choose  $b(x)$ to be the vector of entries $Q_j(x,a^*)$.\\
Let $G'(x)$ be the adjoint matrix of the $q\times q$ matrix built from $G(x)$ by taking only the first $q$ columns. 
Then 
$$G'(x).G(x)=\left(\begin{array}{cc}
\D(x).1\!\!1_q & \star
\end{array}\right).$$
Thus, by multiplying \eqref{eqlin} by $G'(x)$ on the left side, we have
$$\left[\begin{array}{c}\D(x)a^*_1(x)+P_1(a^*_{q+1}(x),\cdots,a^*_p(x))\\
\D(x)a^*_2(x)+P_2(a^*_{q+1}(x),\cdots,a^*_p(x)\\
\vdots\\
\D(x)a^*_q(x)+P_q(a^*_{q+1}(x),\cdots,a^*_p(x))
\end{array}\right]-G'(x).f(x)=\D(x).G'(x).b(x)$$
for some $P_i$ depending linearly on $a^*_{q+1}(x),\cdots,a^*_p(x)$.
Then we set
\begin{equation}\label{a(x)}\begin{split}a_i(x):=a^*_i(x)-c_i(x) \ \text{ for } 1\leq i\leq q,\\
a_i(x):=a^*_i(x) \hspace{1.3cm} \text{ for } q<i\leq p,\end{split}\end{equation}
where $c(x)$ is the vector $G'(x).b(x)$.
Since $G'(x)$ has rank $q$, this shows that
$$G(x).a(x)-f(x)=0$$
i.e. $a(x)$ is a solution of \eqref{system}.\\
\\
Now we have to bound the height and the degree of $a(x)$ in terms of the height and the degree of $a^*$. For simplicity we will bound the height and the degree of $a(x)$ when $\cha(\k)=0$. The bounds in positive characteristic are obtained in the same way and they are  similar (the only difference comes from Theorem \ref{division_sep} - see also Remark \ref{poly}).\\
\\
First by Lemma \ref{inequalities_height} \eqref{6} we have
$$\h(g_{i,j}(x)x_n^k)\leq H_g+kD_g.$$
Let us remind that $d\leq \h(\D)\leq q!qD_g^{q!+q}H_g\leq q^qD_g^{q!+q}H_g$.
Thus by theorem \ref{division_sep} we have (by choosing $\e=1$ for simplicity and since $k<d$):

$$\h(R_{i,j,k,l}(x'))\leq 2^{2^{O(\h(\D)^2)}}D_g^{O(d^4)} (H_g+kD_g+1)\leq 2^{2^{O(q^{2q}  D_g^{2(q!+q)}H_g^2)}},\  $$

$$\h(R'_{i,l}(x'))\leq 2^{2^{O(\h(\D)^2)}}D_f^{O(d^4)}(H_f+1)\leq 2^{2^{O(q^{2q}  D_g^{2(q!+q)}H_g^2)}}D_f^{O(d^4)}(H_f+1),$$

\begin{equation*}\begin{split}\h(Q_{i,j,k}(x))\leq 2^{2^{O(q^{2q}  D_g^{2(q!+q)}H_g^2)}}D_g^{d^4+d^3+6d^2-3d+5}\Deg(\D) (H_g&+kD_g+1)\\
&\leq 2^{2^{O(q^{2q}  D_g^{2(q!+q)}H_g^2)}},\end{split}\end{equation*}

\begin{equation*}\begin{split}\h(Q'_j(x))\leq 2^{2^{O(q^{2q}  D_g^{2(q!+q)}H_g^2)}}D_f^{d^4+d^3+6d^2-3d+5}\Deg&(\D)(H_f+1)\\
\leq 2&^{2^{O(q^{2q}  D_g^{2(q!+q)}H_g^2)}}D_f^{8d^4}(H_f+1),\ \end{split}\end{equation*}

\begin{equation*}\begin{split}\Deg(R_{i,j,k,l}(x'))\leq\h(\D)!D_g^{\h(\D)}\leq (\h(&\D)D_g)^{\h(\D)}\\
\leq (q!q& D_g^{q!+q+1}H_g)^{q!q  D_g^{q!+q}H_g}\leq 2^{2^{O(q^{2q}  D_g^{2(q!+q)}H_g^2)}},\end{split}\end{equation*}

\begin{equation*}\begin{split}\Deg(R'_{i,l}(x'))\leq \h(\D)!D_f^{\h(\D)}\leq(q!q  D_g^{q!+q}D_fH_g)^{q!q  D_g^{q!+q}H_g}&\\
&\leq (2D_f)^{2^{O(q^{2q}  D_g^{2(q!+q)}H_g^2)}}, \end{split}\end{equation*}

\begin{equation*}\begin{split}\Deg(Q_{i,j,k}(x))\leq \h(\D)!D_g^{\h(\D)+1}&\Deg(\D)\\
&\leq (q!q D_g^{q!+q+1}H_g)^{q!q  D_g^{q!+q}H_g+2}D_g\leq 2^{2^{O(q^{2q}  D_g^{2(q!+q)}H_g^2)}},\end{split}\end{equation*}

$$\Deg(Q'_j(x))\leq (q!q  D_g^{q!+q}D_fH_g)^{q!q  D_g^{q!+q}H_g}D_f^{q!q  D_g^{q!+q}H_g}\leq (2D_f)^{2^{O(q^{2q}  D_g^{2(q!+q)}H_g^2)}}.\hspace{2.2cm}$$\\
\\
We set 
$$D_{a^*}:=\Deg(a^*),\ \ H_{a^*}:=\h(a^*).$$
By Lemma \ref{inequalities_height} \eqref{5} we have
$$\h(Q_{i,j,k}(x)a_{i,k}(x'))\leq 2\Deg(Q_{i,j,k}(x))D_{a^*}\max\{\h(Q_{i,j,k}(x')),H_{a^*}\}$$
$$\leq 2^{2^{O(q^{2q}  D_g^{2(q!+q)}H_g^2)}}D_{a^*}H_{a^*}.$$
Moreover 
$$\Deg(Q_{i,j,k}(x)a_{i,k}(x'))\leq 2^{2^{O(q^{2q}  D_g^{2(q!+q)}H_g^2)}}D_{a^*}.$$
Since the components of $b(x)$ are the $Q_j(x,a^*)$ we obtain by \eqref{Qj} and Lemma \ref{inequalities_height} \eqref{2} 

\begin{equation*}\begin{split}\h(b(x))\leq (pd+1)\left(2^{2^{O(q^{2q}  D_g^{2(q!+q)}H_g^2)}}D_{a^*}\right)^{pd}&\max_j\{\Deg(Q'_j(x))\}\times \\
\max&\left\{2^{2^{O(q^{2q}  D_g^{2(q!+q)}H_g^2)}}D_{a^*}H_{a^*} ,\h(Q'_j(x))\right\}.\end{split}\end{equation*}
Since $d\leq \h(\D)\leq q!q  D_g^{q!+q}H_g$ we get
\begin{equation}\label{hb}\h(b(x))\leq (2^{p+1}D_f)^{2^{O(q^{2q}  D_g^{2(q!+q)}H_g^2)}}D_{a^*}^{pd+1}\max\{H_{a^*},(H_f+1)\}. \end{equation}
Moreover \eqref{Qj} gives
\begin{equation}\label{Db}\Deg(b(x))\leq (2^{p+1}D_f)^{2^{O(q^{2q}  D_g^{2(q!+q)}H_g^2)}}D_{a^*}.\end{equation}
We have
$$\h(\D)\leq q.q!D_g^{q!+q}H_g$$
and, by Lemma \ref{inequalities_height} \eqref{2} and \eqref{5} the height any $(q-1)\times(q-1)$ minor of $G$ is bounded by
$$(q-1)!D_g^{(q-1)!}((q-1)D_g^{q-1}H_g)\leq q.q!D_g^{q!+q}H_g. $$
Thus, by Lemma \ref{inequalities_height} \eqref{5}, the height of the  coefficients of $G'(x)$ is less than

$$ 2D_g^2q.q!D_g^{q!+q}H_g=2q!qD_g^{q!+q+2}H_g.$$
Hence, by Lemma \ref{inequalities_height} \eqref{2} and \eqref{5},  using \eqref{hb}, \eqref{Db} and since $\Deg(G'(x))\leq D_g$ we obtain

\begin{equation*}\begin{split}\h(G'(x).b(x))\leq q(\Deg(G'(x))&\Deg(b(x)))^q\times\\
&\big(2\Deg(G'(x))\Deg(b(x))\max\{\h(G'(x)),\h(b(x))\} \big)\end{split}\end{equation*}
$$\leq (2^{p+1}D_f)^{2^{O(q^{2q}  D_g^{2(q!+q)}H_g^2)}}D_{a^*}^{q+pd+1}\max\{H_{a^*},(H_f+1)\}.$$
Hence, by \eqref{a(x)} and Lemma \ref{inequalities_height} \eqref{2}
$$\h(a(x))\leq  (2D_f+2^p)^{2^{O(q^{2q}  D_g^{2(q!+q)}H_g^2)}}D_{a^*}^{2H_gp+3}\max\{H_{a^*},(H_f+1)\}.$$
Moreover
$$\Deg(a(x))\leq D_{a^*}\Deg(b(x))\leq (2^{p+1}D_f)^{2^{O(q^{2q}  D_g^{2(q!+q)}H_g^2)}}D_{a^*}^2.$$
\\
Let $\Phi:=q^{2q}  D_g^{2(q!+q)}H_g^2$. By the inductive assumption we can find a solution $a'(x')=(a_{i,k}(x'))_{1\leq i\leq p, \ 0\leq k\leq d-1}$ of the system \eqref{eqn-1} such that
$$\h(a'(x'))\leq C_1\left(n-1,qd,pd,2^{2^{O(\Phi)}},2^{2^{O(\Phi)}},(2D_f)^{2^{O(\Phi)}}\right)\cdot D_f^{O(\Phi^2)}2^{2^{O(\Phi)}}(H_f+1),$$
$$\Deg(a'(x'))\leq C_2\left(n-1,qd,pd,2^{2^{O(\Phi)}},2^{2^{O(\Phi)}},(2D_f)^{2^{O(\Phi)}}\right).$$
Since
$$D_{a^*}\leq \Deg(a'(x')),$$
$$H_{a^*}\leq d\cdot\Deg(a'(x'))^d(\h(a'(x'))+d-1)$$
by \eqref{a^*} and Lemma \ref{inequalities_height} \eqref{2},
 the solution $a(x)$ of \eqref{system} satisfies
\begin{equation*}\begin{split}\h(a(x))\leq  (2D_f+2^p)^{2^{O(\Phi)}}&\times \\
C_2\left(n-1,qd,pd,2^{2^{O(\Phi)}},\right.&\left.2^{2^{O(\Phi)}},(2D_f)^{2^{O(\Phi)}}\right)^{2H_gp+3}\times \\
C_1\left(n-1,qd,pd,2^{2^{O(\Phi)}},2^{2^{O(\Phi)}},\right.&\left.(2D_f)^{2^{O(\Phi)}}\right)\cdot D_f^{O(\Phi^2)}2^{2^{O(\Phi)}}(H_f+1).\end{split}\end{equation*}

$$\Deg(a(x))\leq  (2^{p+1}D_f)^{2^{O(\Phi)}}C_2\left(n-1,qd,pd,2^{2^{O(\Phi)}},2^{2^{O(\Phi)}},(2D_f)^{2^{O(\Phi)}}\right)^2.$$
Then the result is proven with $\Phi=q^{2q}  D_1^{2(q!+q)}H_1^2$,
\begin{equation*}\begin{split}C_1(n,q,p,H_1,D_1,D_2)=(2D_2+2^p)^{2^{O(\Phi)}}&\times \\
C_2\left(n-1,qd,pd,2^{2^{O(\Phi)}},\right.&\left.2^{2^{O(\Phi)}},(2D_2)^{2^{O(\Phi)}}\right)^{2H_1p+3}\times \\
C_1\left(n-1,qd,pd,2^{2^{O(\Phi)}},2^{2^{O(\Phi)}},\right.&\left.(2D_2)^{2^{O(\Phi)}}\right)\cdot D_2^{O(\Phi^2)}2^{2^{O(\Phi)}}\end{split}\end{equation*}
and
\begin{equation*}\begin{split}C_2(n,q,p,H_1,D_1,D_2)=(&2^{p+1}D_2)^{2^{O(\Phi)}}\times\\
C_2&\left(n-1,qd,pd,2^{2^{O(\Phi)}},2^{2^{O(\Phi)}},(2D_2)^{2^{O(\Phi)}}\right)^2.\end{split}\end{equation*}
\end{proof}

\begin{rmk}\label{poly}
The proof of this result does not give a nice bound on the functions $C_1\left(n,q,p,H_1,D_1,D_2\right)$ or $C_2(n,q,p,H_1,H_2,D_1,D_2)$. One can check that $C_2(n,q,p,H_1,H_2,D_1,D_2)$ is bounded by a tower of exponentials of length $2n+1$ of the form
$$(2^{p+1}D_2)^{2^{2^{\iddots^{O(qD_1H_1)}}}}.$$
For $C_1\left(n,q,p,H_1,D_1,D_2\right)$ we obtain the same kind of bound.\\
In positive characteristic, the bounds are more complicated and are not polynomial in $D_2$ since the bounds on the complexity of the Weierstrass Division are not polynomial in $D_2$.
\end{rmk}



\section{Proof of Theorem \ref{main}}\label{proof_main}

In this part we will denote by $R_n$ the ring of algebraic power series in $n$ variables over a field $\k$ and $\wdh{R}_n$ its $(x_1,\cdots,x_n)$-adic completion. If $\k$ is a finite field we replace $\k$ by $\k(t)$ where $t$ is transcendental over $\k$ -  this does not change the problem. Thus we may assume that $\k$ is infinite.\\
\\
For any $\k\lag x\rag$-module $M$, we have $\ord_{M}(m)=\ord_{\wdh{M}}(m)$ for all $m\in M$, thus we may assume that $M$ is equal to $R_n^s/N$ for some $R_n$-submodule $N$ of  $R_n^s$.\\
We set $e:=(e_1,\cdots,e_s)$ where the $e_1$, $\cdots$, $e_s$ is the canonical basis of $R_n^s$.  Let us assume that $N$ is  generated by  $L_1(e)$, $\cdots$, $L_l(e)$ where
$$L_i(e)=\sum_{j=1}^sl_{i,j}e_j \ \text{ for } 1\leq i\leq l,$$
and let $H$ (resp. $D$) be a bound on the height (resp. the degree) of the $l_{i,j}$.\\
The proof is done by a double induction on $s$ and $n$. Let 
$$f=f_1e_1+\cdots+f_se_s\in R_n^s\backslash N.$$
We consider the following cases:\\
\\
- (1) If $s=1$ and $N=(0)$, then $M=R_n$ and in this case 
$$\ord_M(f)=\ord_{R_n}(f)\leq \h(f)$$ for any algebraic power series $f$ by Lemma \ref{height_ord}.\\
\\
- (2) Assume that $s=1$ and $N\neq (0)$ is an ideal  of $R_n$. After a linear change of variables there exists  a Weierstrass polynomial $g(x)\in N$ with respect to $x_n$, whose coefficients are in $R_{n-1}$,   of degree $d$ in $x_n$. Then $M$ is isomorphic to $R_{n-1}^{d}/N'$ for some sub-module $N'$ of $R_{n-1}^{d}$. The isomorphism $M\simeq R_{n-1}^d/N'$ is induced by the morphism $R_n\lgw R_{n-1}^d$ sending a power series $f(x)\in R_n$ onto $(r_0,\cdots,r_{d-1})$ where
$$r=r_0+r_1x_n+\cdots+r_{d-1}x_n^{d-1}$$
is 
the remainder of the Weierstrass division of $f(x)$ by $g(x)$. Then $N'$ is the $R_{n-1}$-sub-module of $R_{n-1}^d$ generated by the vectors of coefficients of the remainders of the Weierstrass division of the elements of $M$ by $g(x)$.\\
If $f(x)\in R_n$ then the remainder $r$ of the  division of $f$ by $g$  has height less than $C_1\cdot(\h(f)+1)$ for some $C_1>0$ depending only on $g(x)$ and $\Deg(f)$ (by Theorem \ref{division_sep} - moreover $C_1$ is polynomial in $\Deg(f)$ when $\cha(\k)=0$). We remark that $f$ and $r$ have the same image in $M$. If $r=r_0+r_1x_n+\cdots+r_{d-1}x_n^{d-1}$, with $r_i\in R_{n-1}$ for all $i$, then  $(r_0,r_1,\cdots,r_{d-1})$ has height  less that $C_1\cdot(\h(f)+1)$ again by Theorem \ref{division_sep}. Moreover $\ord_M(f)=\ord_M(r)$. Since $x_n$ is integral over $R_{n-1}$, there exists a constant $a>0$ such that $x_n^a\in (x')$, with $x'=(x_1,\cdots,x_{n-1})$. Thus $(x)^{ac}\subset (x')^c$ for any integer $c$. So  we have: \\
$$\ord_{R_{n-1}^{d}/N'}(r)=\sup\{c\in\N\ / \ r\in(x')^cR_{n-1}^{d}/N'\}\geq\frac{1}{a+1}\ord_{M}(r).$$
By the induction hypothesis on $n$ there exists  $C>0$ such that 
$$\ord_{R_{n-1}^{d}/N'}(r)\leq C\cdot\h(r)\ \ \ \ \forall r\in R^{d}_{n-1}.$$ Thus we have
$$\ord_M(f)=\ord_M(r)\leq (a+1)\ord_{R_{n-1}^{d}/N'}(r)\leq (a+1)C\,\h(r)\leq(a+1) CC_1\,(\h(f)+1).$$
If $\cha(\k)=0$ and $C$ is assumed to depend polynomially on $\Deg(r)$ by the induction hypothesis, then $(a+1)CC_1$ depends polynomially on $\Deg(f)$ by Theorem \ref{division_sep}.\\
\\
\noindent - (3) Assume that $s\geq 2$ and $f_s$ is in the ideal of $R_n$ generated by $l_{1,s}$, $\cdots$, $l_{\kappa,s}$. \\
Then we can write
$$f_s=a_1l_{1,s}+\cdots+a_{\kappa} l_{\kappa,s}$$
where the $a_i$ are algebraic power series with $\h(a_i)\leq C_2\cdot(\h(f_s)+1)$ for all $i$ and  $C_2>0$  depends only on the $l_{i,s}$ and $\Deg(f_s)$ (by Theorem \ref{IMP_alg}). Moreover, when $\cha(\k)=0$, $C_2$ depends polynomially on $\Deg(f_s)\leq \Deg(f)$ by Remark \ref{poly}. Let us set
$$f':=f-\sum_{i=1}^{\kappa}a_iL_i(e).$$
We set $N'=N\cap (R_n^{s-1}\times\{0\})$. We denote by $M'$ the sub-module of $M$ equal to 
$$\frac{R_n^{s-1}\times\{0\}}{N'}.$$
By Artin-Rees Lemma there exists a constant $c_0>0$ such that

$$(x)^{c+c_0}M\cap M'\subset (x)^cM' \ \ \forall c\in\N.$$
Hence we have
$$\ord_M(f)=\ord_M(f')\leq \ord_{M'}(f')+c_0.$$
By the induction hypothesis on $s$, there exists $C'>0$ depending on $\Deg(f')$ (thus on $\Deg(f)$ by Theorem \ref{IMP_alg}) such that 
$$\ord_{M'}(f')\leq C'\cdot\h(f').$$ 
If $\cha(\k)=0$ and we assumed that $C'$ depends polynomially on $\Deg(f')$ by the induction hypothesis, then $C'$ depends polynomially on $\Deg(f)$ by Remark \ref{poly}. Hence
$$\ord_M(f)\leq \ord_{M'}(f')+c_0\leq C'\cdot\h(f')+c_0\leq (C'+c_0)\h(f')$$
and $C'+c_0$ depends polynomially on $\Deg(f)$ in characteristic zero.
\\
\\
- (4) Assume that $s\geq 2$ and $f_s$ is not in the ideal of $R_n$ generated by $l_{1,s}$, $\cdots$, $l_{\kappa,s}$. \\
Then by the case $s=1$, there exists  $C>0$ depending only on the $l_{i,s}$ and  $\Deg(f_s)$ such that 
$$\ord_{R_n}(f_s+a_1l_{1,s}+\cdots+a_{\kappa}l_{\kappa,s})\leq C\cdot\h(f_s)$$
 for every $a_{i}\in R_n$. Moreover $C$ depends polynomially on $\Deg(f_s)\leq \Deg(f)$ when $\cha(\k)=0$. 
 Let us remark that for every $f\in R_n^s$ we have
 $$\ord_M(f)=\sup\{k\ /\ f\in (x)^kM\}=\sup\{k\ / \ f\in (x)^kR_n^s \text{ modulo }N\}$$
 $$=\sup\{k\ / \ \exists a_1,\cdots, a_{\kappa}\in R_n,\ f+a_1L_1(e)+\cdots+a_{\kappa}L_{\kappa}(e)\in (x)^k R_n^s\}$$
 $$=\sup_{a_1,\cdots,a_{\kappa}\in R_n}\left\{\ord_{R_n^s}(f+a_1L_1(e)+\cdots+a_{\kappa}L_{\kappa}(e))\right\}.$$
 Thus 
$$\ord_M(f)\leq \sup_{a_1,\cdots,a_{\kappa}\in R_n}\left\{\ord_{R_n}(f_s+a_1l_{1,s}+\cdots+a_{\kappa}l_{\kappa,s})\right\}\leq C\cdot\h(f_s)\leq C\cdot\h(f)$$
since 
$$\ord_{R_n^s}(g)=\min_{i=1,\cdots,s}\{\ord_{R_n}(g_i)\}\leq \ord_{R_n}(g_s)$$
for every $g=(g_1,\cdots,g_s)=g_1e_1+\cdots+g_se_s\in R_n^s$.


  
  \section{Proof of Theorem \ref{main2}}\label{proof_main2}

Let $I$ be an ideal of $\wdh{R}_n$.
    We set $J:=I\cap \k[x]$.
  We have the following lemma:
  
  \begin{lemma}\label{height_ideal}
  We have $\het(I)\geq\het(J)$ and if
 $I$ is generated by algebraic power series then $\het(I)=\het(J)$. \\
 On the other hand  if $I$ is the intersection of a finite number of ideals which are powers of prime ideals of the same heights, i.e.
 $$I=P_1^{n_1}\cap\cdots\cap P_l^{n_l} \text{ for some primes } P_i \text{ with } \het(P_i)=\het(P_j)\ \  \forall i, j,$$
then the equality $\het(I)=\het(J)$ implies that $I$ is generated by algebraic power series.
  \end{lemma}
  
\begin{proof}
We have 
$$I=\k\lb x \rb \Longleftrightarrow J=\k[x].$$ 
Thus we may assume that $I$ and $J$ are proper ideals. In this case $J\subset (x)\k[x]$ so $\het(J)=\het(J\k[x]_{(x)})$. Since the morphism $\k[x]_{(x)}\lgw \k\lb x\rb$ is faithfully flat  $\het(J\k[x]_{(x)})=\het(J\k\lb x\rb)$. Then $\het(J)\leq \het(I)$ because $J\k\lb x\rb\subset I$.\\
\\
Let us assume that $I$ is generated by algebraic power series. By Noetherianity there exists a finite number of algebraic power series $a_1$,..., $a_r\in\k\lag x\rag$ that generate $I$.  Since $\k\lag x\rag$ is the Henselization of $\k[x]_{(x)}$, there exists an \'etale map $\k[x]_{(x)}\lgw A$ where $A$ is a local ring such that $\k[x]_{(x)}\lgw \k\lag x\rag$ factors through $\k[x]_{(x)}\lgw A$ and $a_1$,..., $a_r$ are images of elements $a'_1$,..., $a'_r\in A$. By faithful  flatness of $A\lgw \k\lag x\rag$ we have $\het((a'_1,\cdots,a'_r)\cdot A)=\het(I)$. Since the morphism $\k[x]_{(x)}/J\lgw A/(a'_1,\cdots,a'_r)$ is a localization of a  finite injective morphism, we get $\dim(\k[x]_{(x)}/J)=\dim(A/(a'_1,\cdots,a'_r)\cdot A)$, so $\het(J)=\het((a'_1,\cdots,a'_r)\cdot A)=\het(I)$.\\
\\
Now we assume that $\het(I)=\het(J)$.\\
First we consider the case where $I$ is a prime ideal. Then $J$ is also a prime ideal. If $\het(J)=\het(I)$, then $\het(J\k\lb x\rb)=\het(I)$ and since $J\k\lb x\rb\subset I$, then $I$ is a prime associated to $J\k\lb x\rb$. Since $J$ is radical, then $J\k\lag x\rag$ is also a radical ideal: indeed since $\k[x]/J$ is reduced, then its completion $\k\lb x\rb/J\k\lb x\rb$ is also reduced (see (1) p. 180 of \cite{HS}) so $\k\lag x\rag/J\k\lag x\rag$ is reduced. If $J\k\lag x\rag=P'_1\cap\cdots \cap P'_r$ is a prime decomposition of $J\k\lag x\rag$, then the ideals $P'_i\k\lb x\rb$ are prime ideals by Lemma 5.1 \cite{KPPRM} so 
$$J\k\lb x\rb=P'_1\k\lb x\rb\cap \cdots\cap P'_r\k\lb x\rb$$
is a prime decomposition of $J\k\lb x\rb$ and $I$ is equal to one of the $P'_i\k\lb x\rb$, let us say $P'_1\k\lb x\rb=I$. In particular $I$ is generated by algebraic power series. \\
Now let us assume that $I=P_1\cap \cdots\cap P_l$ where the $P_i$ are prime ideals of the same height. Let $J_i:=P_i\cap\k[x]$. Then $J=J_1\cap \cdots\cap J_l$. Since $\het(J)\leq \het(J_i)\leq \het(P_i)=\het(I)=\het(J)$ for every $i$ we have  $\het(J_i)=\het(P_i)$ for all $i$, thus $P_i$ is generated by algebraic power series by the previous case, thus $I$ is also generated by algebraic power series.\\
Finally let us assume that  $I=P_1^{n_1}\cap\cdots\cap P_l^{n_l}$ where the $P_i$ are  prime ideals of  the same height and the $n_i$ are positive integers. Let us set $J_i=P_i\cap \k[x]$. Then $P_i^{n_i}\cap \k[x]$ is an ideal containing $J_i^{n_i}$ whose radical is $J_i$. So $\sqrt{J}=J_1\cap\cdots\cap J_l$. Since $\het(\sqrt{J})=\het(J)=\het(I)=\het(\sqrt{I})$ then $\sqrt{I}$ is generated by algebraic power series by the previous case. Thus the associated primes ideals of $\sqrt{I}$, i.e. the $P_i$, are generated by algebraic power series. Hence the $P_i^{n_i}$ are generated by algebraic power series and $I$ also.

\end{proof}

\noindent
From now on we assume that 
 $$I=P_1^{n_1}\cap\cdots\cap P_l^{n_l} \text{ for some primes } P_i \text{ with } \het(P_i)=\het(P_j)\ \  \forall i, j$$
 and $R$ denotes the ring $\k\lb x\rb/I$.\\
  Let $\k[x]_d$ be the set of polynomials of degree $\leq d$ and  $J_d:=J\cap \k[x]_d$ for every integer $d$. 
  We set  for every integer $d\geq 0$:
  $$\Phi(d):=\dim_{\k}\left(\frac{\k[x]_d}{J_d}\right).$$
  The function $d\lgm\Phi(d)$ coincides with a polynomial function of degree $p:=\dim\left(\frac{\k[x]}{J}\right)=n-$ht$(J)$ for $d$ large enough. Then we define for every integer $d\geq 0$:
  $$\Psi(d):=\dim_{\k}\left(\frac{R}{(x)^d}\right).$$
  The function $d\lgm \Psi(d)$ coincides with a polynomial function of degree $q:=\dim(R)=n-$ht$(I)$ for $d$ large enough.  So $\Psi(d^p)$ and $\Phi(d^q)$ are polynomial functions of same degree (equal to $pq$) for $d$ large enough. By choosing $a>0$ large enough the leading coefficient of $\Phi(ad^q)$ will be strictly greater than the leading coefficient of $\Psi(d^p)$. Thus   for such  a constant $a>0$ we have
  $$\Psi(d^p)<\Phi(ad^q) \ \ \ \forall d>>0.$$
  This means that the canonical $\k$-linear map 
  $$\frac{\k[x]_{ad^q}}{J_{ad^q}}\lgw\frac{R}{(x)^{d^p}}$$
  is not injective for $d$ large enough. For every $d$ large enough let $p_d$ be a non-zero element of the kernel of this map. By assumption there exists a constant $C$ such that 
  $$\ord_R(p_d)\leq C\cdot\deg(p_d)\leq Cad^q\ \ \ \forall d.$$
  Since $p_d$ is in the kernel of the previous $\k$-linear map, we have $\ord_R(p_d)\geq d^p$, thus
  $$Cad^q\geq d^p.$$
  But such an inequality is satisfied (for some constant $a>0$) if and only if $q\geq p$, i.e. if $\dim(R)\geq \dim\left(\frac{\k[x]}{J}\right)$. This last inequality is equivalent to ht$(I)\leq $ ht$(J)$. Thus, by Lemma \ref{height_lemma}, such an inequality is satisfied if and only ht$(I)= $ ht$(J)$, i.e. if and only if $I$ is generated by algebraic power series. This proves  Theorem \ref{main2}.
  
  
  \section{An example}\label{contex}
  Here we show  through an example that Lemma \ref{height_ideal} and Theorem \ref{main2} are not true in general.\\
  \\
  Let $\k=\C$ and $n=3$. For simplicity we  denote the variables $x_1$, $x_2$, $x_3$ by $x$, $y$, $z$. We set 
  $$f(z):=-\log(1-z)=\sum_{k\geq 1}\frac{1}{k}z^k.$$ Let $Q=(x,y)^2=(x^2,y^2,xy)$ and $Q'=Q+(x+f(z)y)$ be ideals of $\C\lb x,y,z\rb$. Then 
  $$\sqrt{Q}=\sqrt{Q'}=(x,y).$$\\
  (1) \emph{$Q'$ is not generated by algebraic power series but $\het(Q'\cap \C[x,y,z])=\het(Q')=2$:}\\
\\
We have $(x,y)^2=Q\subsetneq Q'$ since $x+f(z)y\notin (x,y)^2$, but there is no algebraic power series $g(x,y,z)$ such that
$$x+f(z)y=g(x,y,z) \text{ modulo } Q.$$
Indeed, if it were the case, by replacing $x^2$, $y^2$ and $xy$ by zero in the expansion of $g$, we would find an algebraic power series $h(z)$ such that
$$x+f(z)y=x+h(z)y$$
which is not possible since $f(z)$ is transcendental. So $Q'\cap\C[ x,y,z]=(x,y)^2$ and $Q'$ is not generated by algebraic power series. Since $(x,y)^2\subset Q'\subset (x,y)$, we have $\het(Q')=2=\het((x,y)^2)=\het(Q'\cap\C[x,y,z])$. This proves the claim.\\
\\
(2) \emph{$A'=\C\lb x,y,z\rb/Q'$ satisfies the local zero estimate \eqref{lze2}  of Corollary \ref{izu}: 
\begin{equation}\label{lze_ex}\ord_{A'}(p)\leq 2\deg(p)\ \ \ \ \forall p\in\C[x,y,z]\backslash Q'.\end{equation}}\\
\\
  Since $Q\subset Q'\subset (x,y)$ we have the canonical quotient morphisms:
  $$A:=\frac{\C\lb x,y,z\rb}{Q}\lgw A':=\frac{\C\lb x,y,z\rb}{Q'}\lgw B:=\frac{\C\lb x,y,z\rb}{(x,y)}.$$
\\
We consider two cases: \\
\\
(a) if $p$ is a polynomial of $\k[x,y,z]$, $p\notin (x,y)$, then we have $\ord_{A'}(p)\leq \ord_B(p)$. But we claim that $\ord_B(p)\leq \deg(p)$. Indeed, let $f\in\C\lb x,y,z\rb$ be equal to $p$ modulo $(x,y)$. Since $p\notin (x,y)$, $p$ has a nonzero monomial of the form $az^k$ for some $a\in\C$ and $k\leq \deg(p)$. Since $f-p\in(x,y)$, then $f$ has also a non zero monomial $az^k$. So $\ord_{\C\lb x,y,z\rb}(f)\leq k$.\\
Thus we have
    \begin{equation}\label{111}\ord_{A'}(p)\leq \deg(p)\ \ \ \forall p\in\C[x,y,z]\backslash (x,y).\end{equation}\\
(b)   Now let $p$ be a polynomial with $p\in (x,y)$ but $p\notin Q'$. In particular $p\neq 0$. Then there exists a unique polynomial $p'$ of the form
  $$p'=a(z)x+b(z)y$$
   where $a(z)$, $b(z)\in\k[z]$, $\deg(p')\leq \deg(p)$ and 
   $$p'\equiv p \text{ modulo }Q.$$
   Let $n$ be an integer such that $n+1\leq \ord_{A'}(p)=\ord_{A'}(p')$. This means that 
   $$p'\in(x,y,z)^{n+1}+(x,y)^2+(x+f(z)y),$$ thus 
   $$p'=\e+\eta+c\cdot(x+f(z)y)$$
   for some $\e\in(x,y,z)^{n+1}$, $\eta\in (x,y)^2$ and $c\in \C\lb x,y,z\rb$, . Since $p'=a(z)x+b(z)y$ we obtain
   \begin{equation}\label{penible}\begin{split}
   (a(z)-c(0,0,z))x+(b(z)-c(0,0,z)&f(z))y\\
   =\eta+(c-c&(0,0,z))(x+f(z)y)+\e\end{split}\end{equation}
   But $\eta':=\eta+(c-c(0,0,z))(x+f(z)y)\in (x,y)^2$.
   Moreover  $\e$ can be written as 
   $$\e=\e'(x,y,z)+\e_x(z)x+\e_y(z)y+\e_1(z)$$
    where $\e'(x,y,z)\in(x,y)^2$, $\e_x(z)$, $\e_y(z)\in(z)^{n}\C\lb z\rb$ and $\e_1(z)\in(z)^{n+1}\C\lb z\rb$.
   Thus \eqref{penible} shows that 
   $$   a(z)-c(0,0,z)=\e_x,\ \ \ \ b(z)-c(0,0,z)f(z)=\e_y(z),$$
   $$\eta'+\e'=0,\ \ \ \ \e_1(z)=0.$$
  This proves that if $\ord_{A'}(p)\geq n+1$  then there exists 
    $c(z)\in\C\lb z\rb$
such that
  $$\ord_z(a(z)-c(z))\geq n\text{ and } \ord_z(b(z)-c(z)f(z))\geq n.$$
Let us write 
$$a(z)=\sum_ka_kz^k,\ \ b(z)=\sum_{k}b_kz^k\text{ and } c(z)=\sum_{k}c_kz^k.$$
If $\deg(p)\leq d$ for some integer $d$, then $\deg(a)$, $\deg(b) \leq d-1$ thus
$$a_k=b_k=0\ \ \ \forall k\geq d.$$
 Since $\ord_z(a(z)-c(z))\geq n$ then 
 $$a_k=c_k\ \ \ \forall k<n.$$
In particular, if $d<n$, we have
$$b_{d}=\cdots=b_{n-1}=c_{d}=\cdots=c_{n-1}=0.$$
Since $\ord_z(b(z)-c(z)f(z))\geq n$ then
$$\left[\begin{array}{ccccc}  0 & 0 & 0 & 0 & 0\\
1 & 0 & 0 & 0 & 0\\
\frac{1}{2} & 1 &0 & 0 & 0\\
\vdots & \vdots & \ddots & \ddots & \vdots\\
\frac{1}{n-1} & \frac{1}{n-2} & \cdots & 1 & 0
  \end{array}\right]\left[\begin{array}{c}c_0\\ c_1 \\ c_2 \\ \vdots \\ c_{n-1}\end{array}\right]= \left[\begin{array}{c}b_0\\ b_1 \\ b_2 \\ \vdots \\ b_{n-1}\end{array}\right].$$
  There are two cases to be considered: either $\ord_{A'}(p)\leq \deg(p)$ and the local zero estimate \eqref{lze_ex} is satisfied, either $\ord_{A'}(p)>\deg(d)$. In the latter case we can choose $n\geq d$.
  In particular 
  \begin{equation}\label{hankel}\left[\begin{array}{ccccc}  
\frac{1}{d} & \frac{1}{d-1} &  \frac{1}{d-2} & \cdots &  1\\
\frac{1}{d+1} & \frac{1}{d} & \cdots & \cdots & \frac{1}{2}\\
\vdots & \vdots & \ddots & \ddots & \vdots\\
\frac{1}{n-1} & \frac{1}{n-2} & \cdots & \cdots & \frac{1}{n-d}
  \end{array}\right]\left[\begin{array}{c}c_0\\ c_1  \\ \vdots \\ c_{d-1}\end{array}\right]= \left[\begin{array}{c}b_d\\ b_{d+1} \\ \vdots \\ b_{n-1}\end{array}\right]=\left[\begin{array}{c}0\\ 0 \\ \vdots \\ 0\end{array}\right].\end{equation}
  Let us assume that the local zero estimate \eqref{lze_ex} is not satisfied, i.e. $\ord_{A'}(p)>2\deg(p)$. Then we can choose $n=2d$. But for $n=2d$ the matrix  
  $$\left[\begin{array}{cccc}  
\frac{1}{d} & \frac{1}{d-1}  & \cdots &  1\\
\frac{1}{d+1} & \frac{1}{d} & \cdots & \frac{1}{2}\\
\vdots & \vdots & \ddots  & \vdots\\
\frac{1}{2d-1} & \frac{1}{2d-2} & \cdots  & \frac{1}{d}
  \end{array}\right]$$
  is a Hilbert matrix and is not singular. This means that Equation \eqref{hankel} for $n=2d$ has no nontrivial solution, hence $c_0=\cdots=c_{n-1}=0$. This proves that $a_k=b_k=0$ for every $k$ which contradicts the assumption that $p\neq 0$. This proves that for every polynomial $p\in \C[x,y,z]$, $p\notin Q'$, we have
      \begin{equation}\label{222}\ord_{A'}(p)\leq 2\deg(p).\end{equation}
      

  \section{Grauert-Hironaka-Galligo Division of power series}\label{GHG}
  Let $\l$ be a linear form on $\R^n$ with positive coefficients. Let us consider the following order on $\N^n$:
for all $\a$, $\b\in\N^n$, we say that $\a\leq \b$ if 
$$(\l(\a),\a_1,\cdots,\a_n)\leq_{lex}(\l(\b),\b_1,\cdots,\b_n)$$
 where  $\leq_{lex}$ is the lexicographic order. This order induces an order on the set of monomials $x_1^{\a_1}\cdots x_n^{\a_n}$: we set $x^{\a}\leq x^{\b}$ if $\a\leq \b$. This order is called the \emph{monomial order induced by $\l$}.
If 
$$f:=\sum_{\a\in\N^n}f_{\a}x^{\a}\in\k\lb x\rb,$$ the \emph{initial exponent} of $f$ with respect to the previous order is
$$\exp(f):=\min\{\a\in\N^n\ / \ f_{\a}\neq 0\}=\inf \Supp(f)$$ 
where the \emph{support} of $f$ is  $\Supp(f):=\{\a\in\N^n\ / \ f_{\a}\neq 0\}$. The \emph{initial term} of $f$ is $f_{\exp(f)}x^{\exp(f)}$ and is denoted by $\ini(f)$. This is the smallest non-zero monomial in the  expansion of $f$ with respect to the previous order.\\
\\
Let $g_1$, $\cdots$, $g_s$ be elements of $\k\lb x\rb$.
Set
$$\Delta_1:=\exp(g_1)+\N^n\text{ and } \ \Delta_i=(\exp(g_i)+\N^n)\backslash \bigcup_{1\leq j <i}\Delta_j, \text{ for } 2\leq i\leq s.$$
Finally, set
$$\Delta_0:=\N^n\backslash\bigcup_{i=1}^s\Delta_i.$$
We have the following theorem:
\begin{theorem}\label{div}\cite{Gr}\cite{Hir}\cite{Ga}
Set $f\in\k\lb x\rb$. Then there exist some unique power series  $q_1$, $\cdots$, $q_s$, $r\in\k\lb x\rb$ such that
$$f=g_1q_1+\cdots+g_sq_s+r$$
$$\exp(g_i)+\Supp(q_i)\subset \D_i \text{ and } \Supp(r)\subset \Delta_0.$$
The power series $r$ is called the \emph{remainder} of the division of $f$ by $g_1$, $\cdots$, $g_s$ with respect to the given monomial order.\\
Moreover if $\k$ is a valued field and $f$, $g_1$, $\cdots$, $g_s$ are convergent power series, then the $q_i$ and $r$ are convergent power series.
\end{theorem}

The uniqueness of the division comes from the fact the $\D_i$'s are disjoint subsets of $\N^n$. The existence  of such decomposition in the formal case is proven through the division algorithm:\\
\\
Set $\a:=\exp(g)$. Then there exists an integer $i_1$ such that $\a\in \D_{i_1}$.\\
$\bullet$ If $i_1=0$, then set $r^{(1)}:=\ini(g)$ and $q_i^{(1)}:=0$ for any $i$.\\
$\bullet$  If $i_1\geq 1$, then set $r^{(1)}:=0$, $q_i^{(1)}:=0$ for $i\neq i_1$ and $q_{i_1}^{(1)}:=\frac{\ini(g)}{\ini(g_{i_1})}$.\\
 Finally set $\displaystyle g^{(1)}:=g-\sum_{i=1}^sg_iq_i^{(1)}-r^{(1)}$. Thus we have $\exp(g^{(1)})>\exp(g)$. Then we replace $g$ by $g^{(1)}$ and we repeat the preceding process.\\
  In this way we construct a sequence $(g^{(k)})_k$ of power series such that, for any $k\in\N$, $\exp(g^{(k+1)})>\exp(g^{(k)})$ and $\displaystyle g^{(k)}=g-\sum_{i=1}^sg_iq_i^{(k)}-r^{(k)}$ with
  $$\exp(g_i)+\Supp(q_i^{(k)})\subset \D_i \text{ and } \Supp(r^{(k)})\subset \Delta_0.$$
  At the limit $k\lgw\infty$ we obtain the desired decomposition.\\
 \\
  But in general if $f$ and the $g_i$ are algebraic power series (or even polynomials) then $r$ and the $q_i$ are not algebraic power series as shown by the following example:
    
  \begin{ex}[Kashiwara-Gabber's Example]\label{K-G}(\cite{Hir} p. 75)
 Let us perform the division of $xy$ by $$g:=(x-y^a)(y-x^a)=xy-x^{a+1}-y^{a+1}+x^{a}y^a$$ as formal power series in $\k\lb x,y\rb$ with an integer $a>1$ (here we choose a monomial order induced by the linear form $\l(\a_1,\a_2)=\a_1+\a_2$). By symmetry the remainder of this division can be written $r(x,y):=s(x)+s(y)$ where $s(x)$ is a formal power series. By  substituting $y$ by $x^a$ we get 
 
 $$s(x^a)+s(x)-x^{a+1}=0.$$
 This relation yields the expansion
 $$s(x)=\sum_{i=0}^{\infty}(-1)^ix^{(a+1)a^i}.$$
Thus the remainder of the division has Hadamard gaps and thus is not algebraic if $\cha(\k)=0$. Hadamard gaps are defined as follows: 
 \end{ex}
 
\begin{definition}
Let $x=(x_1,\cdots,x_n)$. A power series $f=\sum_kf_k$ where $f_k$ is a homogeneous polynomial of degree $k$ for every $k$ has Hadamard gaps if the indices $n_1<n_2<n_3<\cdots$ of all non-zero homogeneous terms of $f$ satisfy the condition $n_{k+1}>C n_k$ for all $k$ where $C>1$.\\
Over a characteristic zero field, a power series having Hadamard gaps cannot be algebraic.

\end{definition}

 \begin{ex}\label{ex2}
 Let $\k$ be a field of any characteristic.
 Set 
 $$f_n:=xy-\sum_{i=0}^n(-1)^ix^{(a+1)a^i}.$$ Then by the previous example
 $$f_n\equiv \sum_{i>n}(-1)^ix^{(a+1)a^i} \text{ mod. } (g).$$
 Thus $$\ord_{\k\lb x\rb/(g)}(f_n)\geq (a+1)a^{n+1}.$$
 Since $f_n$ is a polynomial of degree $(a+1)a^n$, this shows that the bound of Corollary \ref{izu} is optimal.
 \end{ex}
 
  
  \section{Generic Kashiwara-Gabber Example}\label{KG_gen}
  In this part we will investigate a particular case of division. Mainly we will consider the problem of dividing an algebraic power series $f(x,y)$ in two variables by an algebraic power series $g(x,y)$ whose initial term is equal to $xy$ with respect to a given monomial order as defined in the previous part. In this case the remainder of the division is the sum $R(x)+S(y)$ of one power series in $x$ and one power series in $y$.
  
    \begin{definition}\label{D-finite}
  Let $\k$ be a characteristic zero field and $x$ be a single variable.
  A \emph{$D$-finite power series} $f$ is a formal power series in $\k\lb x\rb$ satisfying a linear differential equation with polynomial coefficients, i.e. there exist $D\in\N$ and $a_{j}(x)\in\k[x]$ (not all equal to 0) for $0\leq j\leq D$  such that
  $$a_Df^{(D)}+a_{D-1}f^{(D-1)}+\cdots+a_0f=0.$$
  
  \end{definition}
  \noindent
  Let us mention that by \cite{St} any algebraic power series is $D$-finite.\\
  In Example \ref{K-G}, if $\cha(\k)=0$, the remainder is not $D$-finite  since $D$-finite power series have no Hadamard gaps (see \cite{St} or \cite{LR} for instance). We will show that the situation of Example \ref{K-G} is generic in some sense.\\
  \\
  Set 
$$g_{\und{a}}(x,y)=xy-\sum_{(i,j)\in E}a_{i,j}x^iy^j$$
where $\und{a}$ denotes the vector of entries $a_{i,j}\in \k$ for some field $\k$ and $E$ is a finite subset of $\N^2$ such that:
\begin{enumerate}
\item  $(0,0)$, $(0,1)$, $(1,0)$ and $(1,1)\notin E$,
\item $\{(2,0),(0,2)\}\not\subset E$.
\end{enumerate}
If $(0,2)\notin E$, let us choose the linear form $\l$ defined by $\l(e_1,e_2)=3e_1+2e_2$. Then for any $e=(e_1,e_2)\in E$ we have
$\l(e)=3e_1+2e_2>\l(1,1)=5$ since only three situations may occur:\\
- either $e_1\geq 2 $ so $\l(e)\geq 6$,\\
- either $e_1=1$ and $e_2\geq 2$ so $\l(e)\geq 7$, \\
- either $e_1=0$ and $e_2\geq 3$ so $\l(e)\geq 6$. \\
This means that there exists a monomial order induced by a linear form such that $xy$ is the initial term of $g_{\und{a}}(x,y)$. By symmetry this is also true if $(2,0)\notin E$. From now on we fix such monomial order and we perform the division of $xy$ by $g_{\und{a}}(x,y)$:
$$xy=g_{\und{a}}(x,y)Q_{\und{a}}(x,y)+R_{\und{a}}(x)+S_{\und{a}}(y).$$
\begin{lemma}\label{lemma_division}
Let $\k=\Q(\und a)$ where $\und a$ is the set of new undeterminates $a_{i,j}$ for $(i,j)\in E$. Then  $R_{\und a}(x)$ (resp. $S_{\und a}(y)$, $Q_{\und a}(x,y)$) is a power series with coefficients in $\Q[\und a]$.\\
 In particular if $\k$ is a characteristic zero field and $\und \a \in \k^{Card(E)}$ is a vector of elements $\a_{i,j}\in \k$ for every $(i,j)\in E$, then the coefficients of $R_{\und \a}(x)$ (resp. $S_{\und \a}(y)$, $Q_{\und \a}(x,y)$) are those of $R_{\und a}(x)$ (resp. $S_{\und a}(y)$, $Q_{\und a}(x,y)$) evaluated in $\und \a$.
\end{lemma}
\begin{proof}
Since the coefficient of the leading term $xy$ of $xy-\sum_{(i,j)\in E}a_{i,j}x^iy^j$ is equal to 1, we see directly from the division algorithm given in Section \ref{GHG} that the coefficients of $R_{\und a}(x)$, $S_{\und a}(y)$ and $Q_{\und a}(x,y)$ are in $\Q[\und a]$.\\
Then by evaluating the terms of the equality 
$$xy=g_{\und{a}}(x,y)Q_{\und{a}}(x,y)+R_{\und{a}}(x)+S_{\und{a}}(y)$$
in $\und a$ we necessarily obtain the equality
$$xy=g_{\und{\a}}(x,y)Q_{\und{\a}}(x,y)+R_{\und{\a}}(x)+S_{\und{\a}}(y)$$
by unicity of the division.
\end{proof}
For every $k\in\N\backslash\{0,1\}$ we set $$E_k=\{(0,k+1),(k+1,0),(k,k)\}.$$
We have the following result:

\begin{prop}\label{prop2}
Let $E$ be a finite set as before such that $E_k\subset E$ for some integer $k>1$. Let $(\a_{i,j})\in\C^{Card(E)}$ whose coordinates are algebraically independent over $\Q$. Then $R_{\und{\a}}(x)$ is not a $D$-finite power series. In particular this is not an algebraic power series.
\end{prop}

\begin{proof}
Let $N=Card(E)$. The proof is made by induction on $N$.\\
If $N=3$, then $E=E_k$. If $\a_{0,k+1},\a_{k+1,0},\a_{k,k}\in \C$ are algebraically independent over $\Q$ and $R(x):=R_{\und{\a}}(x)$ is a $D$-finite power series, then $R(x)$ satisfies a differential equation:

\begin{equation}\label{rel_diff}P_d(x)R^{(d)}(x)+\cdots+P_1(x)R(x)+P_0(x)=0\end{equation}
where $P_1(x)$, $\cdots$, $P_d(x)\in\C[x]$. If we expand this relation in terms of a $\Q(\und{\a})$-basis of the $\Q(\und{\a})$-vector space $\C$, we obtain at least one non-trivial relation of the same type where the $P_i(x)$ are in $\Q(\und{\a})[x]$. So we assume that $P_i(x)\in \Q(\und{\a})[x]$ for all $i$ and even $P_i(x)\in \Q[\und{\a}][x]$ for all $i$ by multiplying this relation by a common denominator of the coefficients of the $P_i$.
Since $\a_{k+1,0}$, $\a_{0,k+1}$ and $\a_{k,k}$ are algebraically independent over $\Q$, we are reduced to assume that $R_{a,b,c}(x)$ is $D$-finite over $\Q[a,b,c]$ where $a$, $b$ and $c$ are new indeterminates and
$R_{a,b,c}(x)$ is the $x$-depending part of the remainder of the division of $xy$ by $xy-ax^{k+1}-by^{k+1}-cx^ky^k$:
$$xy=\left(xy-ax^{k+1}-by^{k+1}-cx^ky^k\right)Q_{a,b,c}(x,y)+R_{a,b,c}(x)+S_{a,b,c}(y).$$
By Lemma \ref{lemma_division} $R_{a,b,c}(x)\in\Q[a,b,c]\lb x\rb$, $S_{a,b,c}(y)\in\Q[a,b,c]\lb y\rb$ and for every point $\a=(\a_{0,k+1},\a_{k+1,0},\a_{k,k})\in\C^3$, the power series $R_{\a}(x)$ and $S_{\a}(y)$ are equal to $R_{a,b,c}(x)$ and $S_{a,b,c}(y)$ evaluated in $\a$.\\
We may assume that the polynomials $P_i=P(a,b,c,x)$, coefficients of the Relation \eqref{rel_diff}, are globally coprime, otherwise we factor out their common divisor. For $0\leq i\leq d$, let $V_i$ be the subvariety of $\C^3$ which is the zero locus of the coefficients of $P_i(a,b,c,x)$ (seen as a polynomial in $x$). Let $V$ be the intersection of $V_0$, $\cdots$, $V_d$. Then if $(\und{\a})\notin V$, one of the $P_i(\und{\a},x)$ is non-zero and $R_{\und{\a}}(x)$ is $D$-finite over $\C[x]$. Since we have assumed that the $P_i(a,b,c,x)$ are globally coprime, $V$ is a finite union of algebraic curves and points, except if all but one $P_i$ are equal to 0. In this latter case, we have $P_d(a,b,c,x)R^{(d)}_{a,b,c}(x)=0$ which means that $R^{(d)}_{a,b,c}(x)=0$, thus we may replace $P_d$ by 1 and in this case $V=\emptyset$.  \\
\\
From now on we replace $c$ by $-ab$ and we have the relation:
\begin{equation}\label{*}xy=(x-by^k)(y-ax^k)Q_{a,b,-ab}(x,y)+R_{a,b,-ab}(x)+S_{a,b,-ab}(y).\end{equation}
By symmetry we have $R_{b,a,-ab}(y)=S_{a,b,-ab}(y)$.
If we replace $(x,y)$ by $(by,ax)$ in  \eqref{*} we get
$$abxy=ab(y-a^kx^k)(x-b^ky^k)Q_{a,b,-ab}(by,ax)+R_{a,b,-ab}(by)+S_{a,b,-ab}(ax),$$
thus we obtain
\begin{equation}\label{xy}\frac{1}{ab}R_{a,b,-ab}(by)=S_{a^k,b^k,-(ab)^k}(y).\end{equation}
By replacing $y$ by $ax^k$ in \eqref{*} we obtain:
$$ax^{k+1}=R_{a,b,-ab}(x)+S_{a,b,-ab}(ax^k)$$
so 
$$a^kx^{k+1}=R_{a^k,b^k,-a^kb^k}(x)+S_{a^k,b^k,-a^kb^k}(a^kx^k)$$
and
\begin{equation}\label{ind_r}a^kx^{k+1}=R_{a^k,b^k,-a^kb^k}(x)+\frac{1}{ab}R_{a,b,-ab}(a^kbx^k)\end{equation}
by \eqref{xy}.
By writing
$$R_{a,b,-ab}(x)=\sum_{l\geq 1} r_l(a,b)x^l$$
and plugging it in \eqref{ind_r}
we obtain 
$$ r_l(a,b)=0\ \ \ \forall l\leq k\  \text{ and } \ r_{k+1}(a^k,b^k)=a^k.$$
Moreover the coefficient of $x^{kl}$on both sides of \eqref{ind_r}, for every $l\geq 1$, is equal to
$$0=r_{kl}(a^k,b^k)+\frac{1}{ab}r_l(a,b)a^{kl}b^l$$
hence
$$r_{kl}(a^k,b^k)=-r_l(a,b)a^{kl-1}b^{l-1}.$$
Thus 
$$r_{k+1}(a,b)=a,\ r_{k(k+1)}(a,b)=-a^{k+1}b,\ r_{k^2(k+1)}(a,b)=a^{k(k+1)+1}b^{k+1}$$
and by induction 
$$r_{k^i(k+1)}(a,b)=(-1)^ia^{\sum_{j=0}^{i}k^j}b^{\sum_{j=0}^{i-1}k^j}\ \ \forall i\geq 1$$
$$=(-1)^ia^{\frac{k^{i+1}-1}{k-1}}b^{\frac{k^i-1}{k-1}}$$
and $r_l(a,b)=0$ if $\frac{l}{k+1}$ is not a power of $k$.
Thus we obtain
$$R_{a,b,-ab}(x)=\sum_{i=0}^{\infty}(-1)^{i}a^{\frac{k^{i+1}-1}{k-1}}b^{\frac{k^i-1}{k-1}}x^{(k+1)k^i}.$$
Exactly as in the example of Kashiwara-Gabber, this shows that $R_{\a,\b,-\a\b}(x)$ is not  $D$-finite if $\a\b\neq 0$.\\
Let $S\subset\C^3$ be the surface of equation $ab+c=0$.  In particular $S$ is not included in $V$ since the components of $V$ have dimension $\leq 1$. Then we see that for any $(\a,\b,\g)\in S\backslash \{ab=0\}$, $R_{\a,\b,\g}(x)$ is not $D$-finite. This contradicts the assumption that $R_{a,b,c}(x)$ is $D$-finite since we have shown that this would imply that $R_{\a,\b,\g}(x)$ is $D$-finite for every $(\a,\b,\g) \notin V$. Thus $R_{a,b,c}(x)$ is not $D$-finite.\\
\\
Let us assume that $N>3$ and that the proposition is proven for every set of cardinal $N-1$ containing $E_k$. Let us assume that $R_{\und{a}}(x)$ is $D$-finite, i.e. there exist polynomials $P_i\in\C(\und{a})[x]$, for $1\leq i\leq d$, such that
$$P_d(\und{a},x)R^{(d)}_{\und{a}}(x)+\cdots+P_1(\und{a},x)R_{\und{a}}(x)+P_0(\und{a},x)=0.$$
As we did before, we may assume that $P_i\in\Q[\und{a},x]$ for all $i$. By dividing the previous relation by a common divisor of the $P_i$, we may assume that the $P_i$ are globally coprime. For $0\leq i\leq d$ let $V_i$ denote the subvariety of $\C^N$ which is the zero locus of the coefficients of $P_i(x)$ (seen as a polynomial with coefficients in $\Q[\und{a}]$). Let $V$ be the intersection of $V_0$, $\cdots$, $V_d$. As in the previous case, since the $P_i$ are globally coprime, then codim$_{\C^N}(V)\geq 2$.\\
 Let $(i_0,j_0)\in E\backslash E_k$ and set $E'=E\backslash \{(i_0,j_0)\}$. Set $W=\{a_{i_0,j_0}=0\}$; we have codim$_{\C^N}(W)=1$. By the inductive assumption, $R_{\und{\a}}(x)$ is not $D$-finite for every $\und{\a}\in W$ such that tr.deg$_{\Q}\Q(\und{\a})=N-1$.
But if $\und{\a}\in W\backslash V$ and  tr.deg$_{\Q}\Q(\und{\a})=N-1$ (we may find such an $\und{\a}$ since codim$_{\C^N}(V)$ is strictly larger than codim$_{\C^N}(W)$), we see that $R_{\und{\a}}(x)$ is not $D$-finite which is a contradiction since $\und{\a}\notin V$. Thus $R_{\und{a}}(x)$ is not $D$-finite and the proposition is proven for sets $E$ of cardinal $N$.\\
\end{proof}

\begin{ex}
If $E$ does not contain any of the sets $E_k$ for $k>1$ then Proposition \ref{prop2} is no valid in general. For instance let us consider
$$E\subset \{(i,i+j),\  (i,j)\in\N^2,i>0, j>0\}.$$
We set $F=\{(i,j),\ (i,i+j)\in E\}$. Let us consider the Weierstrass division 
$$z=\left[z-\sum_{(i,j)\in F}a_{i,i+j}z^iy^j\right]Q(z,y)+R(y)$$
where $Q$ and $R$ are algebraic power series by Lafon Division Theorem. Then by replacing $z$ by $xy$ we obtain the division of $xy$ by $g_{\und{a}}(x,y)$:
$$xy=\left[xy-\sum_{(i,j)\in E}a_{i,j}x^iy^j\right]Q(xy,y)+R(y).$$
Thus $R_{\und{a}}(x)=R(x)$ is an algebraic power series.
\end{ex}

\begin{ex}
Let $h(x,y)$ and $d(x,y)$ be two algebraic power series over $\C$ and let us assume that the initial term of $d(x,y)$ is $xy$. The division of $h$ by $d$ yields the relation:
$$h(x,y)=d(x,y)Q(x,y)+R(x)+S(y).$$
By Newton-Puiseux Theorem there exist $n\in\N$ and $x(y)\in\C\langle y\rangle$, $y(x)\in\C\langle x\rangle$ such that
$$d(x(y),y^n)=d(x^n,y(x))=0.$$
Thus we obtain
$$h(x(y^{\frac{1}{n}}),y)=R(x(y^{\frac{1}{n}}))+S(y)$$
$$h(x^n,y(x))=R(x^n)+S(y(x)).$$
This yields the relation:
$$R(x^n)-R(x(y(x)^{\frac{1}{n}}))=h(x^n,y(x))-h(x(y(x)^{\frac{1}{n}})).$$
By replacing $x$ by $x^n$ we see that  there exist two algebraic power series $f(x)$ and $g(x)$ such that
$$R(x^{n^2})-R(g(x))=f(x).$$
But this is impossible if $R(x)=e^x$ by Schanuel's conjecture \cite{Ax}. This shows that in general $D$-finite power series (here $e^x$) which are not algebraic are not remainders of such a Weierstrass division.
\end{ex}


\section{Gap Theorem for remainders of division of algebraic power series}\label{Div_alg}
By a Theorem of Schmidt (see Hilfssatz 5 \cite{Sc}) an algebraic power series has no large gaps in its expansion. More precisely his result asserts that if an algebraic power series $f$ is written as $f=\sum_kf_{n(k)}$ where $f_{n(k)}$ is a non-zero homogeneous polynomial of degree $n(k)$ and $(n(k))_k$ is  increasing, then
$$\limsup_{k\lgw\infty}\frac{n(k+1)}{n(k)}<\infty.$$
 We prove here the same result for remainders of the Grauert-Hironaka-Galligo Division, i.e. it does not have more than Hadamard gaps.
%
%
%

  \begin{theorem}\label{remainder_division}
  Let $g_1$, $\cdots$, $g_s\in\k\langle x\rangle$ and let us fix a monomial order induced by a linear form as in Section \ref{GHG}. Then there exists a function $C:\N\lgw \R_{>0}$ such that the following holds:\\
  Let $f\in\k\lag x\rag$ be an algebraic power series  and let $r$ be the remainder of the division of $f$ by $g_1$, $\cdots$, $g_s$ with respect to the given monomial order. Let us write $r=\sum_{k=1}^{\infty}r_{n(k)}$ where $r_h$ is a homogeneous polynomial of degree $h$, $(n(k))_k$ is an increasing sequence of integers and $r_{n(k)}\neq 0$ for any $k\in\N$. Then
  $$n(k+1)\leq C(\Deg(f))\cdot n(k) \ \ \ \ \  \forall k\gg 0.$$
  In particular
  $$\limsup_{k\lgw\infty}\frac{n(k+1)}{n(k)}<\infty.$$

  \end{theorem}
  \vspace{0.2cm}
  \begin{proof}
 Let $I$ denote the ideal generated by $g_1$, $\cdots$, $g_s$.\\
    Let us set $f_k:=f-\sum_{i=1}^kr_{n(i)}$ for every $k\in\N$. The remainder of the division of $f$ by $g_1$, $\cdots$, $g_s$ is $\sum_{i=k+1}^{\infty}r_{n(i)}$, thus 
    $$\ord_{\k\lb x\rb/I}(f_k)=\ord_{\k\lb x\rb/I}\left(\sum_{i=k+1}^{\infty}r_{n(i)}\right)\geq n(k+1).$$
On the other hand by Lemma \ref{inequalities_height} \eqref{3}
    $$\h(f_k)\leq \h(f)+\Deg(f)\cdot n(k)$$
     thus $\h(f_k)\leq 2 \Deg(f)\cdot n(k)$ for $k$ large enough since $(n(k))_k$ is increasing. Hence, by Theorem \ref{main}, and since $\Deg(f_k)=\Deg(f)$, there exists $C'>0$ depending on $\Deg(f)$ such that
    $$\ord_{\k\lb x\rb/I}(f_k)\leq 2C'\cdot\Deg(f)\cdot n(k)$$
    for $k$ large enough. So the theorem is proven with $C=C'\cdot \Deg(f)$.
    
    \end{proof}
    
    \begin{rmk}
    Example \ref{ex2} shows that this result is sharp.
    \end{rmk}

 \end{document}